\documentclass[sts]{imsart}

\RequirePackage{amsthm,amsmath,amsfonts,amssymb}
\RequirePackage[authoryear]{natbib}
\RequirePackage[colorlinks,citecolor=blue,urlcolor=blue]{hyperref}%% uncomment this for coloring bibliography citations and linked URLs
\RequirePackage{graphicx}%% uncomment this for including figures
\RequirePackage{xcolor}
\RequirePackage{float,booktabs,multirow,graphicx}
\RequirePackage[OT1]{fontenc}
\RequirePackage{algorithm,algpseudocode}
\RequirePackage{xr-hyper}
\usepackage{pdfpages}
\startlocaldefs
\theoremstyle{plain}
\newtheorem{theorem}{Theorem}
\newtheorem{lemma}{Lemma}
\newtheorem{proposition}{Proposition}

\newtheorem{assumption}{Assumption}
\theoremstyle{remark}
\newtheorem{remark}{Remark}
\newtheorem{example}{Example}
\def\t{^\top}

\def\cG{\mathcal{G}}

\def\cM{\mathcal{M}}

\def\cX{\mathcal{X}}
\def\cY{\mathcal{Y}}
\def\cU{\mathcal{U}}
\def\real{\mathcal{R}}

\def\vec{\mathrm{vec}}

\def\supp{\mathrm{supp}}

\def\ind{\mathbb{I}}

\def\pr{P}
\def\id{\iota}
\DeclareMathOperator*{\argmin}{arg\,min}

\endlocaldefs

\begin{document}

\begin{frontmatter}
\title{A Geometric Perspective on Bayesian and Generalized Fiducial Inference}
%\title{A sample article title with some additional note\thanksref{T1}}
\runtitle{Geometry of Inference}
%\thankstext{T1}{A sample of additional note to the title.}

\begin{aug}
  \author[A]{\fnms{Yang}~\snm{Liu}\ead[label=e1]{yliu87@umd.edu}\orcid{0000-0003-3559-688X}},
  \author[B]{\fnms{Jan}~\snm{Hannig}\ead[label=e2]{jan.hannig@unc.edu}\orcid{0000-0002-4164-0173}}
\and
\author[C]{\fnms{Alexander C.}~\snm{Murph}\ead[label=e3]{murph@lanl.gov}\orcid{0000-0001-7170-867X}}
%%%%%%%%%%%%%%%%%%%%%%%%%%%%%%%%%%%%%%%%%%%%%%
%% Addresses                                %%
%%%%%%%%%%%%%%%%%%%%%%%%%%%%%%%%%%%%%%%%%%%%%%
\address[A]{Yang Liu is an Associate Professor from the Department of Human Development and Quantitative Methodology at University of Maryland, College Park, MD, USA\printead[presep={\ }]{e1}.}
\address[B]{Jan Hannig is a Professor from the Department of Statistics and Operations Research at the University of North Carolina, Chapel Hill, NC, USA\printead[presep={\ }]{e2}.}
\address[C]{Alexander C. Murph is a postdoctoral researcher from the Computer, Computational, and Statistical Sciences Division (CCS-6) at the Los Alamos National Laboratory, Los Alamos, NM, USA\printead[presep={\ }]{e3}.}
\end{aug}

\begin{abstract}
Post-data statistical inference concerns making probability statements about model parameters conditional on observed data. When {\it a priori} knowledge about parameters is available, post-data inference can be conveniently made from Bayesian posteriors. In the absence of prior information, we may still rely on objective Bayes or generalized fiducial inference (GFI). Inspired by approximate Bayesian computation, we propose a novel characterization of post-data inference with the aid of differential geometry. Under suitable smoothness conditions, we establish that Bayesian posteriors and generalized fiducial distributions (GFDs) can be respectively characterized by absolutely continuous distributions supported on the same differentiable manifold: The manifold is uniquely determined by the observed data and the data generating equation of the fitted model. Our geometric analysis not only sheds light on the connection and distinction between Bayesian inference and GFI, but also allows us to sample from posteriors and GFDs using manifold Markov chain Monte Carlo algorithms. A repeated measures analysis of variance example is presented to illustrate the sampling procedure.
\end{abstract}


\end{frontmatter}

\section{Introduction}
\label{s:intro}

A post-data probability represents the degree of belief or plausibility that a certain assertion about model parameters is true given the observed data, which differs from a classical frequentist (i.e., pre-data) probability that is attached to the generative process of the observed data \citep{Dempster1964, MartinLiu2015b}. Post-data statistical inferences are most commonly made from a Bayesian posterior that is jointly determined by the prior distribution of model parameters and the likelihood function of the model \citep[e.g.,][Section 1.3]{GelmanEtAl2013}. When little \textit{a priori} information about parameters can be garnered, we may still resort to default or weakly informative priors to make Bayesian inference \citep{KassWasserman1996, Berger2006, Berger2015}.

Alternatively, we can avoid prior specification altogether and obtain a post-data probability distribution of parameters by inverting the data generating process. This idea originated from Fisher's fiducial argument \citep{Fisher1925, Fisher1930, Fisher1933, Fisher1935} and motivated the development of Dempster-Shafer theory \citep{Dempster1966, Dempster1968, Dempster2008}, inferential models \citep{MartinLiu2013, MartinLiu2015a, MartinLiu2015b, MartinLiu2015c}, generalized fiducial inference \citep[GFI;][]{CisewskiHannig2012, Hannig2009, Hannig2013, HannigEtAl2016, LaiEtAl2015, LiuHannig2016, LiuHannig2017, MurphEtAl2021, ShiEtAl2021} and so forth. Among all the descendents of Fisher's fiducial inference, only GFI is considered in the present paper; the associated post-data distribution of parameters is referred to as the \textit{generalized fiducial distribution (GFD)}.

A statistical model specifies how data are generated through a \textit{data generating equation (DGE)}, which is a function of parameters and random components with completely known distributions (e.g., uniform or standard Gaussian variates).\footnote{A data generating equation may be referred to as a data generating algorithm \citep[DGA;][]{MurphEtAl2021} when the generative process rather than the formal mathematical expression is of interest.} The DGE plays a key role in approximating post-data inference by simulation \citep{CranmerEtAl2020}. When a proper prior can be specified, we may simulate parameters and random components independently, obtain imputed data through the DGE, and retain the samples if and only if the imputed and observed data are sufficiently close. Such an accept-reject scheme is often referred to as approximate Bayesian computation \citep[ABC; e.g.,][]{Beaumont2019, BeaumontEtAl2002, BeaumontEtAl2009, Blum2010, FearnheadEtAl2012, MarinEtAl2012, SissonFan2011, SissonEtAl2018}: The retained samples of parameters approximately follow the posterior distribution and hence can be utilized to estimate posterior expectations. If no prior distribution is available, we can still sample random components but not parameters. To circumvent the latter, GFI proceeds to pair each realization of random components with the optimal parameter values such that the resulting imputed data is as close to the observed data as possible in some sense. Indeed such a best matching to the observed data may still not be good enough: Those values are deemed incompatible with the observed data and therefore have to be discarded, leading to a rejection step similar to ABC. It turns out that the resulting marginal samples of parameters approximately follow the GFD \citep{HannigEtAl2016}. 

%Although GFI has been successfully applied to many statistical problems such as linear mixed effects models \citep{CisewskiHannig2012}, high-dimensional regression \citep{LaiEtAl2015}, binary and ordinal logistic response models \citep{LiuHannig2016, LiuHannig2017}, and covariance estimation \citep{ShiEtAl2021}, it remains difficult in general to design efficient algorithms to sample from the fiducial distribution as pointwise evaluation of the fiducial density is not always viable. 

%As the most up-to-date result along this line, \citet[][Theorem 1]{HannigEtAl2016} derived a closed-form expression of the fiducial density assuming that the random components and observed data have the same dimensionality (in addition to several other regularity conditions). However, their result does not apply to models with additional random effects, e.g., mixed-effects models.

It is then natural to ponder what the limits of the truncated distributions are when we request the imputed data to be infinitesimally close to the observed data in approximate post-data inference. As the main result of the present work, we completely characterize the weak limit for both approximate Bayesian inference and GFI when the truncation set contracts to a twice continuously differentiable submanifold of the joint space of parameters and random components. We are able to express the absolutely continuous densities of the limiting distributions with respect to the intrinsic measure of the submanifold, and show that Bayesian posteriors and GFDs in the usual sense are the corresponding marginals on the parameter space (Propositions \ref{prop:bayes} and \ref{prop:gfi}). As a contribution to the literature of GFI, we derive an explicit formula for the fiducial density in Proposition \ref{prop:gfi} that is more general compared to Theorem 1 of \citet{HannigEtAl2016}. Meanwhile, our work should be distinguished from \citet{MurphEtAl2022}, which also studied the geometry of GFI but focused on the case when the parameter space itself is a manifold. On the theoretical side, our geometric formulation applies to a broad class of parametric statistical models for continuous data and facilitates insightful comparisons between Bayesian inference and GFI. On the practical side, the geometric characterization suggests an alternative sampling scheme for approximate post-data inference: We apply manifold Markov chain Monte Carlo (MCMC) algorithms \citep[e.g.,][]{BrubakerEtAl2012, ZappaEtAl2018} to sample from the limiting distributions on the data generating manifold and only retain the parameter marginals. For certain problems (e.g., GFI for mixed-effects models), manifold MCMC sampling may scale up better than existing computational procedures.

The rest of the paper is organized as follows. We revisit in Section \ref{s:abc2gfi} the formal definitions of ABC and GFI; a graphical illustration is provided using a Gaussian location example. In Section \ref{s:geom}, we first present a general result (Theorem \ref{thm:gen}): When an ambient distribution is truncated to a sequence of increasingly finer approximations to a smooth manifold, the weak limit is absolutely continuous with respect to the manifold's intrinsic measure. We then apply the general result to derive representations for Bayesian posteriors and GFDs (Propositions \ref{prop:bayes} and \ref{prop:gfi}) and comment on their discrepancies. We review in Section \ref{s:mcmc} an MCMC algorithm that (approximately) samples from distributions on differentiable manifolds. A repeated measures analysis of variance (ANOVA) example is then presented to illustrate the sampling procedure (Section \ref{s:ex}). Limitations and possible extensions of the proposed method are discussed at the end (Section \ref{s:end}).

\section{Approximate Inference by Simulation}
\label{s:abc2gfi}

\subsection{Data Generating Equation}
\label{ss:dge}
Let $\cY$, $\Upsilon$, and $\Theta$ denote the spaces of data, random components, and parameters associated with a fixed family of parametric models: In particular, $\cY\subseteq\real^n$, $\Upsilon\subseteq\real^m$, and $\Theta\subseteq\real^q$, where $n$, $m$, and $q$ are positive integers. Following \citet{HannigEtAl2016}, we characterize the model of interest by its DGE
\begin{equation}
  Y = G(U, \theta),
  \label{eq:dge}
\end{equation}
in which the random components $U\in\Upsilon$ follow a completely known distribution (typically uniform or standard Gaussian), $\theta\in\Theta$ denotes the parameters, and $Y\in\cY$ denotes the random data. (\ref{eq:dge}) can be conceived as a formalization of the data generating code: Given true parameters $\theta$ and an instance of random components $U = u$, a unique set of data $Y = y$ can be imputed by evaluating the DGE, i.e., $y = G(u, \theta)$.

Now suppose that we have observed $Y = y$. Post-data inference aims to assign probabilities to assertions about parameters $\theta$ conditional on the observed data $y$ \citep{MartinLiu2015c}. In the conventional Bayesian framework, we presume that $\theta$ follows a proper prior distribution and make probabilistic statements based on the conditional distribution of $\theta$ given $y$. When it is difficult to specify an informative prior, one may still rely on objective priors that reflect paucity of knowledge or information \citep{KassWasserman1996, Berger2006, Berger2015}. We next revisit the definition of a Bayesian posterior through the lens of ABC, as well as \citeauthor{HannigEtAl2016}'s \citeyearpar{HannigEtAl2016} definition of GFD: The latter replaces the prior sampling of parameters in ABC by an optimization problem in the parameter space, which is a natural workaround when no prior information is available.

\subsection{Approximate Bayesian Computation}
\label{ss:abc}
Let $\rho$ denote the density of $U$, and $\pi$ be the prior density of $\theta$; we only restrict to density functions with respect to the Lebesgue measure and assume that random number generation from $\rho$ and $\pi$ is feasible. Given the observed data $y$ and a pre-specified tolerance level $\varepsilon>0$, ABC is a computational procedure that repeatedly executes the following steps: 
\begin{longlist}
  \item sample $U\sim\rho$;
  \item sample $\theta\sim\pi$ independent of $U$;
  \item accept the draws if $\|G(U, \theta) - y\|\le \varepsilon$ and otherwise reject.
\end{longlist}
The above accept-reject sampling scheme constructs a truncated distribution on $\Upsilon\times\Theta$ with the following density:
\begin{equation}
  \pi_\varepsilon(u, \theta; y) \propto \pi(\theta)\rho(u)\,
  \ind_{\{\|G(u, \theta) - y\|\le \varepsilon\}}(u, \theta),
  \label{eq:abcjoint}
\end{equation}
in which $\|\cdot\|$ denotes the $\ell_2$-norm on the data space $\cY$, and $\ind_A$ denotes the indicator function for a set $A$. Integrating out $u$ results in
\begin{equation}
  \pi_\varepsilon(\theta; y) \propto \pi(\theta)
  \pr\{\|G(U, \theta) - y\|\le \varepsilon|\theta\}.
  \label{eq:abcpost}
\end{equation}
Suppose that $Y$ has an absolutely continuous density $f(y|\theta)$ with respect to the Lebesgue measure on $\cY$, and that $y$ is in the interior of $\cY$. (\ref{eq:abcpost}) approximates the posterior 
\begin{equation*}
  \pi(\theta|y) \propto\pi(\theta)f(y|\theta),
\end{equation*}
because
\begin{equation*}
  f(y|\theta) = \lim_{\varepsilon\downarrow 0}\frac{\pr\{\|G(U, \theta) - y\|\le \varepsilon|\theta\}}{\lambda_{\cY}\{w\in\cY:\|w - y\|\le \varepsilon\}}
\end{equation*}
pointwise in $\theta$, where $\lambda_\cY$ denotes the Lebesgue measure on $\cY$, and thus $\pr\{\|G(U, \theta) - y\|\le \varepsilon|\theta\}$ is approximately proportional to $f(y|\theta)$ when $\varepsilon$ is small.

It is recognized that a more general definition of ABC is available in the literature. The accept/rejection step in our introduction corresponds to the use of a bounded uniform kernel supported on $\ell_2$-balls centered around the observed data; other probabilistic kernels can be used and the corresponding limiting results have been established. Readers are referred to \citet{Beaumont2019}, \citet{MarinEtAl2012}, and \citet{SissonEtAl2018} for more comprehensive surveys of ABC.

\subsection{Generalized Fiducial Inference}
\label{ss:afc}
When prior information about $\theta$ is absent, we can no longer sample $\theta\sim\pi$ in Step \romannumeral2) of the ABC recipe. Nevertheless, we are still able to determine whether the imputed random component $U$ can possibly reproduce the observed data $y$ (up to the pre-specified tolerance $\varepsilon$). Let 
\begin{equation}
  \hat\theta(y, U) = \argmin_{\vartheta\in\Theta}\|G(U, \vartheta) - y\|.
  \label{eq:thetahat}
\end{equation}
The rationale of GFI is to pair each $U$ with the parameter values $\hat\theta(y, U)$ such that $G(U, \hat\theta(y, U))$ gives the closest approximation to $y$.\footnote{$\hat\theta(y, u)$ is assumed to uniquely exist for each $u$ (cf. \romannumeral3) in Assumption \ref{as:reg}).} ABC can then be modified into a Monte Carlo recipe for (approximate) GFI once we replace the prior sampling step by setting $\theta$ to $\hat\theta(y, U)$ and leave everything else intact. This modified procedure simulates from a truncated distribution on $\Upsilon$ with density
\begin{equation}
  \psi_\varepsilon(u) \propto \rho(u)\ind_{ \left\{\|G(u, \hat\theta(y, u)) - y\|\le\varepsilon \right\}}(u),
  \label{eq:fidu}
\end{equation}
which further induces a distribution on $\Theta$ via the map $\hat\theta(y, \cdot)$. 

\citet{HannigEtAl2016} went one step further and defined the GFD as the weak limit of $\hat\theta(y, U)$, wherein $U$ follows (\ref{eq:fidu}), as $\varepsilon\downarrow 0$. Assuming $n = m$ and several regularity conditions on the DGE (Assumptions A.1--A.4), \citet{HannigEtAl2016} showed that the density of the GFD can be expressed as 
\begin{align}
  & \psi(\theta; y) \propto f(y|\theta)\cr
  & \cdot\det\left(\nabla_\theta G(\hat u(y, \theta), \theta)\t\nabla_\theta G(\hat u(y, \theta), \theta)\right)^{1/2}
  \label{eq:fid2016}
\end{align}
in which $\hat u(y, \theta)\in\Upsilon$ satisfies $y = G(\hat u(y, \theta), \theta)$, and $\nabla_\theta G(u, \theta)$ denotes the $n\times q$ Jacobian matrix of $G(u, \theta)$ with respect to $\theta$.\footnote{The assumed regularity conditions guarantee that $\hat u(y, \theta)$ uniquely exists, and that the Jacobian matrix is defined and of full column rank.}

(\ref{eq:fid2016}) conveys an empirical Bayesian interpretation of GFI---the determinant term on the right-hand side of (\ref{eq:fid2016}) can be conceived as a (possibly improper) data-dependent prior. Therefore, GFI in general does not comply with the likelihood principle \cite[e.g.,][Section 1.6.4]{Berger1985}. For instance, \citet{HannigEtAl2016} showed that substituting the $\ell_\infty$- and $\ell_1$-norm for the $\ell_2$-norm in (\ref{eq:fidu}) may lead to fiducial densities different from (\ref{eq:fid2016}) when $n > q$. More discussions on the likelihood principle can be found in Section \ref{ss:bayesfid}.

\subsection{An Illustrative Example}
\label{ss:ex}

Consider the Gaussian location model $Y\sim{\cal N}(\mu, 1)$ with the mean parameter $\mu\in\real$. For ease of graphical display, we focus on the transformed parameter $\theta = \Phi(\mu)\in(0, 1)$, where $\Phi(\cdot)$ denotes the distribution function of ${\cal N}(0, 1)$. We express the corresponding DGE as
\begin{equation}
  Y = \Phi^{-1}(U) + \Phi^{-1}(\theta),
  \label{eq:dgeloc}
\end{equation}
in which $U\sim\mathrm{Unif}(0, 1)$, and $\Phi^{-1}$ is the inverse of $\Phi$ (i.e., the standard Gaussian quantile function). The observed data $y$ value is fixed at $-0.5$.

For Bayesian inference, suppose that $\theta$ follows a $\mathrm{Unif}(0, 1)$ prior, which implies a ${\cal N}(0, 1)$ prior for the mean $\mu$. It is straightforward to verify that the posterior density is
\begin{equation}
\pi(\theta|y) = \phi\left( \Phi^{-1}(\theta) \right)^{-1}\sqrt{2}\phi\left(\sqrt{2}(\Phi^{-1}(\theta) - y/2)\right),\kern-5pt
  \label{eq:postex}
\end{equation}
where $\phi(\cdot)$ stands for the standard Gaussian density. Following the ABC recipe, we simulated $U$ and $\theta$ independently from $\mathrm{Unif}(0, 1)$, shown as evenly scattered dots over $\Upsilon\times\Theta = (0, 1)^2$ on the left panel of Figure \ref{fig:abcgfi}. With a tolerance $\varepsilon = 0.05$, only $(u, \theta)\t$ pairs that satisfy $|\Phi^{-1}(u) + \Phi^{-1}(\theta) - (- 0.5)|\le 0.05$ (dark gray colored dots) survive in the accept-reject step. The empirical $\theta$-marginal distribution of the retained draws closely resembles (\ref{eq:postex}). 

\begin{figure*}[!t]
  \centering
  \includegraphics[width=\textwidth]{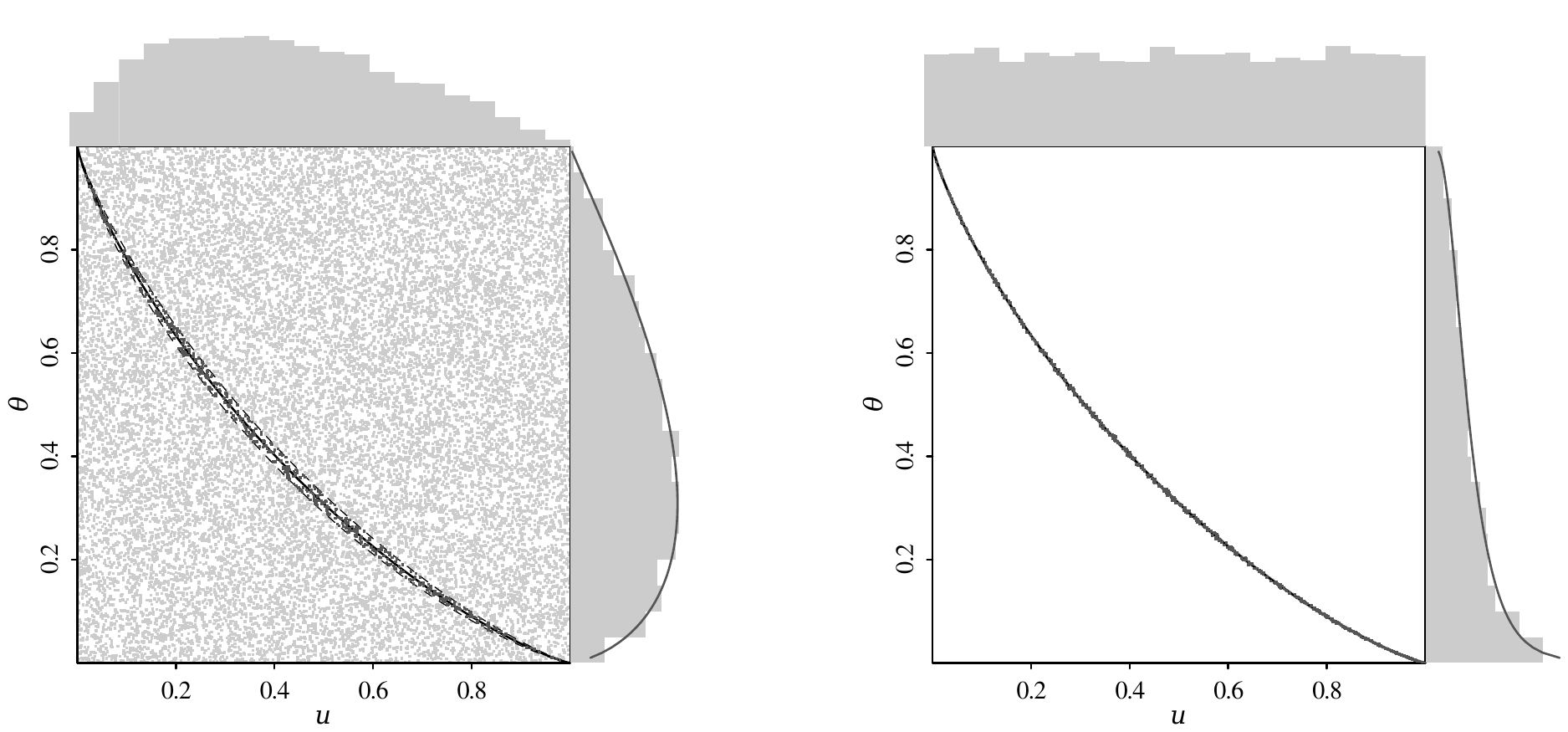}
  \caption{Graphic illustration of the Gaussian location example with $y = -0.5$. Left: approximate Bayesian computation. Samples of random components ($u$) and parameters ($\theta$) are represented as light gray dots in the unit square. $(u, \theta)\t$ pairs that are sufficiently close to the curve $y = \Phi^{-1}(u) + \Phi^{-1}(\theta)$ are kept and highlighted in dark gray (acceptance rate = $2.64\%$). The empirical marginal distributions of the retained samples are displayed as histograms, with the theoretical posterior superimposed on the $\theta$-marginal. Right: fiducial inference. 100\% of the imputed $u$'s are accepted, and each $u$ is paired with $\theta = \Phi(y - \Phi^{-1}(u))$. The empirical marginal distributions of the retained samples are displayed as histograms, with the theoretical fiducial density superimposed on the $\theta$-marginal. Dots fall exactly on the curve but are slightly jittered for clearer visualization.
}
  \label{fig:abcgfi}
\end{figure*}

Meanwhile, the fiducial density (\ref{eq:fid2016}) reduces to\footnote{The normalizing constant is 1.}
\begin{equation}
  \psi(\theta; y) = \phi\left(\Phi^{-1}(\theta)\right)^{-1}\phi\left(y - \Phi^{-1}(\theta)\right).
  \label{eq:fidex}
\end{equation}
For all $u \in (0, 1)$, $\hat\theta(-0.5, u)$ = $\Phi(- 0.5 - \Phi^{-1}(u))$ ensures $|\Phi^{-1}(u) + \Phi^{-1}(\hat\theta(y, u)) - (-0.5)| = 0$. Therefore, all the imputed $u$'s are retained regardless of the value of $\varepsilon$ in the simulation-based fiducial recipe. We associate each $u$ with $\theta = \hat\theta(-0.5, u)$ and plot $(u, \theta)\t$ on the right panel of Figure \ref{fig:abcgfi}. It is observed that the $u$-marginal distribution remains uniform, and (\ref{eq:fidex}) can be well approximated by the histogram of $\theta$.

We learn from the aforementioned illustration that, on the joint space of $u$ and $\theta$, simulation-based Bayesian and fiducial inferences produce distributions that concentrate on 
\begin{equation}
  \cG(y) = \{(u, \theta)\t\in(0, 1)^2: \Phi^{-1}(u) + \Phi^{-1}(\theta) = y\}
  \label{eq:Gex}
\end{equation}
as $\varepsilon\downarrow 0$. $\cG(y)$ collects all the $(u, \theta)\t$ pairs that satisfy the DGE, i.e., (\ref{eq:dgeloc}) with $Y = y$ and $U = u$,  and is geometrically identified as a one-dimensional smooth submanifold embedded in $(0, 1)^2$ (shown as the black solid curve in Figure \ref{fig:abcgfi}). Similar characterizations can be established in a broader class of statistical models for continuous data, which we explicate in the next section.

\section{Geometry of Post-Data Inference}
\label{s:geom}

We have seen in our previous discussion that both the accept-reject ABC and the simulation-based fiducial recipe involve restricting ambient distributions to regions whose sizes are controlled by $\varepsilon$ (see (\ref{eq:abcjoint}) and (\ref{eq:fidu}) for details). We pay heed to the special case that the regions of truncation contract to a twice continuously differentiable submanifold as $\varepsilon\downarrow 0$.

\subsection{General Constraints}
\label{ss:restrict}

Our first result (Theorem \ref{thm:gen}) is completely general: It concerns the weak convergence of a sequence of truncated distributions to a limit that is supported on an implicitly defined submanifold. The proof can be found in Appendix A in the supplementary document.

Let $h:\cX\to\real^n$ be a constraint function, where $\cX$ is an open subset of $\real^d$ and $d > n$. The level set of $h$ at 0 is denoted $\cM = \{x\in\cX: h(x) = 0\}$, and the $\varepsilon$-fattening of $\cM$, where $\varepsilon > 0$, is denoted $\cM^\varepsilon =\{x\in\cX: \|h(x)\| \le\varepsilon\}$. Write $a: \cX\to[0, \infty)$ as an ambient density function.\footnote{Although the function $a$ is not necessarily integrable over the entire ambient space $\cX$, it is referred to as a density function here: Integrable and non-integrable $a$'s are respectively termed as {\it improper} and {\it proper} densities.} Further let $P_\varepsilon$ be the truncation of $a$ to $\cM^\varepsilon$ and is characterized by the density $a(x)\ind_{\cM^\varepsilon}(x)/\int_{\cM^\varepsilon}a(x)dx$.

\begin{assumption}
  Suppose that
  \begin{longlist}
  \item $h$ is a twice continuously differentiable submersion, and thus $\cM$ is a twice continuously differentiable  submanifold of $\cX$, which is equipped with a Riemannian measure $\lambda_\cM$;\footnote{A submersion is a differentiable map, whose differential is surjective at every $x$. The Riemannian measure of the submanifold $\cM$ is induced by the Euclidean metric on the ambient space $\cX$ \citep[][Chapter 13]{Lee2013}.}
  \item $a$ is continuous, $\lambda_\cM\{\supp(a)\cap \cM\} > 0$, and $0 < \int_{\cM^\varepsilon}a(x)dx < \infty$ for all $\varepsilon > 0$;
  \item the collection of probability measures $\{P_\varepsilon: \varepsilon > 0\}$ is tight.
  \end{longlist}
  \label{as:gen}
\end{assumption}

\begin{theorem}
  Under Assumption \ref{as:gen}, $P_\varepsilon\rightsquigarrow P_0$ as $\varepsilon\downarrow 0$, where $P_0$ has the following absolutely continuous density with respect to  $\lambda_\cM$:
  \begin{equation}
    f(x) = \frac{a(x)\det\left(\nabla h(x)\nabla h(x)\t \right)^{-1/2}}{\int_\cM a(w)\det\left(\nabla h(w)\nabla h(w)\t \right)^{-1/2}\lambda_\cM(dw)}\kern-13pt
    \label{eq:dns}
  \end{equation}
  for $x\in\cM$.
  \label{thm:gen}
\end{theorem}
%\begin{remark}
%[YL: Move this remark to the appendix]
%  For a smooth submanifold $\cM\subset\real^d$, we are guaranteed to have a collection of smooth local diffeomorphisms---called \textit{coordinate charts}---that map elements of an open cover of $\cM$ to subsets of a lower-dimensional Euclidean space. Local integration over a set on $\cM$ fixes a coordinate chart that covers this set and performs the usual Lebesgue integration on the Euclidean space pushed forward by the diffeomorphism: The inverse of the diffeomorphic map is often referred as a \textit{local parameterization} of $\cM$. The rescaling for this change of variables depends on the Riemannian metric, which determines a Riemannian measure on the manifold that is invariant to the choice of the coordinate chart. Local integration is extended to global integration via a smooth partition of unity: a collection of weighting functions that allow one piece together the local integrals. Precise definitions of the aforementioned terms can be found in, e.g., \citet{Lee2013}.
%  \label{rmk:dM}
%\end{remark}
\begin{remark}
When a random variable $X$ follows a proper density $a$ in the ambient space $\cX$, (\ref{eq:dns}) can also be deduced as a conditional density of $X$ given $h(X) = 0$ \citep[][Proposition 2]{DiaconisEtAl2013} using the co-area formula (e.g., \citealp{Chavel2006}, Section III.8; \citealp{Federer1996}, Section 3.2.12; \citealp{LelievreEtAl2010}, Lemma 3.2). In this alternative derivation, the denominator of (\ref{eq:dns}) is interpreted as the marginal density of $h(X)$ at 0, which must be finite and positive (see \citealp{DiaconisEtAl2013}, p. 112). Specifically, positivity follows from \romannumeral1) and \romannumeral2) in Assumption \ref{as:gen}), and finiteness is a consequence of tightness, i.e., Assumption \ref{as:gen} \romannumeral3). Details can be found in the proof of Theorem \ref{thm:gen}.
  \label{rmk:coarea}
\end{remark}

\begin{remark}
  Theorem \ref{thm:gen} is inspired by Theorem 3.1 of \citet{Hwang1980}. Hwang's result was proved for a sequence of Gibbs measures that concentrate on the minimum of an energy function. The collection of minimum energy states, or equivalently the limiting manifold, is required to be compact, which is restrictive but often suffices for optimization purposes in statistical physics. In contrast, our result applies to sequentially restricting a known ambient distribution to finer approximations of the data generating manifold---i.e., sublevel sets of $h$, which is often not compact for parametric statistical models.
  \label{rmk:hwang}
\end{remark}

Assumption 1 \romannumeral3), i.e., the tightness of the measures $\{P_\varepsilon\}$, automatically holds if $\cM^\varepsilon$ is compact for sufficiently small $\varepsilon$'s. When all sublevel sets of $h$ are non-compact, however, tightness is determined by the tail behavior of the $\cM^\varepsilon$-restricted probability measures $\{P_\varepsilon\}$. Notably, $a$ being a proper ambient density alone does not guarantee tightness. To illustrate this, we present Example \ref{ex:tight} with a two-dimensional ambient space. It is demonstrated that $\{P_\varepsilon\}$ can still be tight when $a$ is improper but the sublevel set of $h$ tapers off quickly along the first coordinate of $x$ (i.e., $x_1$), and that $\{P_\varepsilon\}$ may not be tight when $a$ is proper but the sublevel set of $h$ rapidly expands as $x_1$ grows.

\begin{figure*}[!t]
  \centering
  \includegraphics[width=\textwidth]{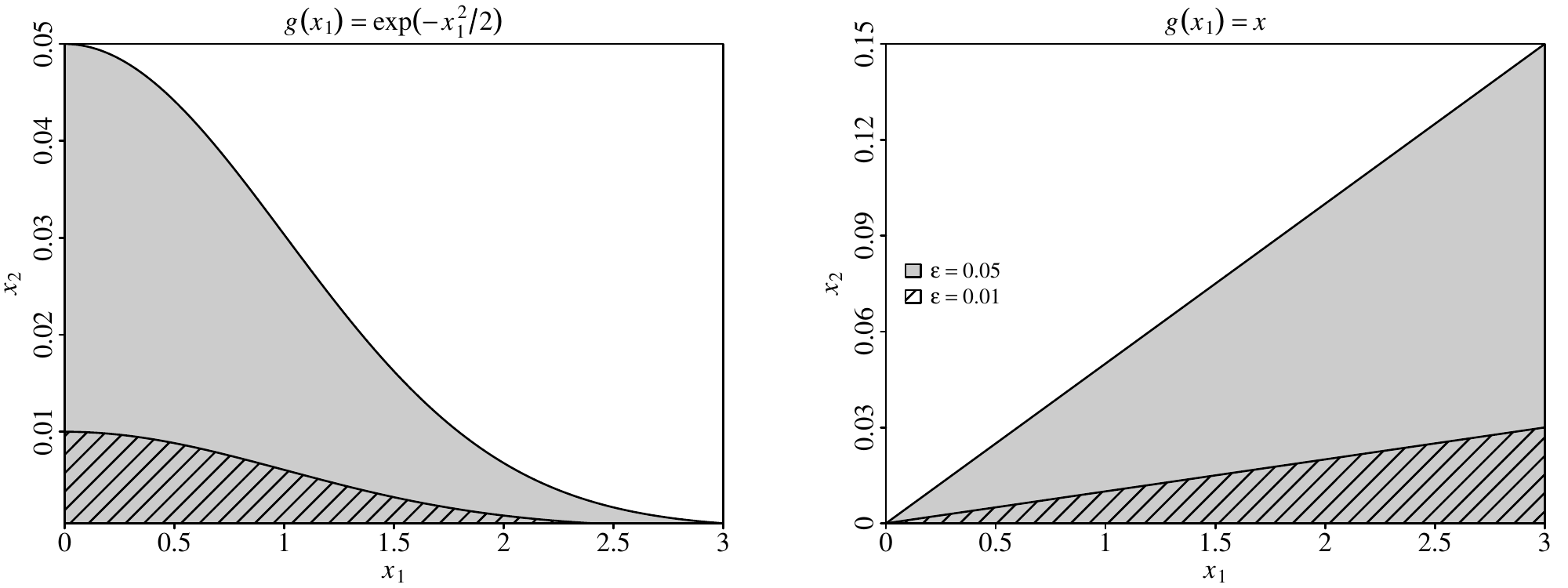}
  \caption{Sublevel sets $\cM^\varepsilon = \{x\in(0, \infty)^2:x_2 \le\varepsilon g(x_1)\}$ in Example \ref{ex:tight}, where $g$ is positive on $(0, \infty)$. Left: $g(x_1) = \exp(-x_1^2/2)$, which vanishes quickly as $x_1\to\infty$. Right: $g(x_1) = x$, which grows to infinity as $x_1\to\infty$. The gray and shaded regions correspond to the sublevel sets when $\varepsilon = 0.05$ and 0.01, respectively.}
  \label{fig:ex1}
\end{figure*}

\begin{example} 
Let $x = (x_1, x_2)\t\in(0, \infty)^2$ and consider the constraint function
\begin{equation}
  h(x) = \frac{x_2}{g(x_1)},
  \label{eq:hx}
\end{equation}
in which $g$ is positive on $(0, \infty)$. The resulting $\varepsilon$-fattened level set is
\begin{equation}
  \cM^\varepsilon = \{x\in(0, \infty)^2: x_2 \le\varepsilon g(x_1)\}.
  \label{eq:meps}
\end{equation}
As $\varepsilon\downarrow 0$, $\cM^\varepsilon\downarrow\cM = \{x\in(0, \infty)^2: x_2 = 0\}$. 

We first set $a(x)\equiv 1$ and $g(x_1) = \exp(-x^2/2)$ (left panel of Figure \ref{fig:ex1}). Even though $a(x)$ is not integrable on the ambient space $(0, \infty)^2$, $g(x_1)$ is integrable on $(0, \infty)$. Hence, $a(x)\ind_{\cM^\varepsilon}(x)/\int_{\cM^\varepsilon} a(x) dx$ is a valid density function that defines the probability measure $P_\varepsilon$. Consider the compact set $K = [0, C]\times [0, 1]$, in which $C>0$. For all $\varepsilon < 1$, 
\begin{align}
  P_\varepsilon\{K\} =\ & \frac{\int_0^C\left[\int_0^{\varepsilon\exp(-x_1^2/2)}dx_2\right]dx_1} {\int_0^\infty\left[\int_0^{\varepsilon\exp(-x_1^2/2)}dx_2\right]dx_1}\cr
  =\ & 2\Phi(C) - 1,
  \label{eq:outside}
\end{align}
which is constant in $\varepsilon$ and can be made arbitrarily close to 1 by choosing a large $C$. So the sequence $\{P_\varepsilon\}$ is tight.

Next, let $a(x) = (1+x_1)^{-2}(1+ x_2)^{-2}$ and $g(x) = x$ (see right panel of Figure \ref{fig:ex1}). $a(x)$ is the joint density of two independent $\hbox{Pareto}(1)$ variates and thus is proper. As $x\to\infty$, the tail probability of the Pareto distribution vanishes linearly while $g(x)$ increases linearly. For all $\varepsilon > 0$, 
\begin{align}
  &\int_{\cM^\varepsilon}a(x)dx\cr
  =\ &\int_0^\infty\left[\int_0^{\varepsilon x_1}(1 + x_2)^{-2}dx_2\right](1 + x_1)^{-2} dx_1\cr
  =\ & \frac{-\varepsilon(1-\varepsilon+\log\varepsilon)}{(1-\varepsilon)^2}.
  \label{eq:denom}
\end{align}
Consider the compact set $K = [0, C]^2$. Then for all $\varepsilon < 1$,
\begin{align}
  &\int_{\cM^\varepsilon\cap K^\complement}a(x)dx\cr
  =\ & \int_C^\infty\left[\int_0^{\varepsilon x_1}  (1+x_2)^{-2}dx_2\right](1+x_1)^{-2} dx_1\cr
  =\ & \frac{-\varepsilon\left[ 1-\varepsilon -(1+C)\log( \frac{1+C\varepsilon}{\varepsilon+C\varepsilon} ) \right]}{(1+C)(1-\varepsilon)^2},
  \label{eq:numer}
\end{align}
in which $K^\complement$ denotes the complement of $K$. The ratio of (\ref{eq:numer}) and (\ref{eq:denom}) gives the probability of $P_\varepsilon\{K^\complement\}$: As $\varepsilon\downarrow 0$,
\begin{align}
  & \frac{\int_{\cM^\varepsilon\cap K^\complement}a(x)dx}{\int_{\cM^\varepsilon}a(x)dx}\cr
  =\ & \frac{1-\varepsilon-(1+C)\log( \frac{1+C\varepsilon}{\varepsilon+C\varepsilon} )}{(1+C)(1-\varepsilon+\log\varepsilon)} \to 1.
  \label{eq:lopital}
\end{align}
As such, the truncated sequence $\{P_\varepsilon\}$ eventually places all the mass outside $K$ for all $C$ and thus cannot be tight.
  \label{ex:tight}
\end{example}

\subsection{Data Generating Manifold}
\label{ss:dgm}
Given a general DGE $G: \Upsilon\times\Theta\to\cY$, observed data $y\in\cY$, and an $\varepsilon> 0$, let
\begin{align}
  \cG^\varepsilon(y) =\ & \{(u\t, \theta\t)\t\in\Upsilon\times\Theta:\cr
  &\ \|G(u, \theta) - y\|\le\varepsilon\},
  \label{eq:Geps}
\end{align}
and its set-theoretic limit
\begin{equation}
  \cG(y) = \{(u\t, \theta\t)\t\in\Upsilon\times\Theta: G(u, \theta) = y\}.
  \label{eq:G0}
\end{equation}
In general, $\cG(y)$ may or may not have a positive Lebesgue measure on $\real^{m + q}$. A further special case of the latter is of interest to us---when $\cG(y)$ is a submanifold of $\Upsilon\times\Theta\subseteq\real^{m + q}$. In this case, we call $\cG(y)$ and $\cG^\varepsilon(y)$ the \textit{data generating manifold} and its \textit{$\varepsilon$-fattening}, respectively. Also denote the \textit{$u$-projections} of $\cG^\varepsilon(y)$ and $\cG(y)$ by
\begin{equation}
  \cU^\varepsilon(y) = \Big\{u\in\Upsilon: \min_{\vartheta\in\Theta}\|G(u, \vartheta) - y\|\le\varepsilon\Big\},
  \label{eq:Ueps}
\end{equation}
and
\begin{equation}
  \cU(y) = \Big\{u\in\Upsilon:\min_{\vartheta\in\Theta}\|G(u, \vartheta) - y\| = 0\Big\},
  \label{eq:U0}
\end{equation}
respectively. $\cG^\varepsilon(y)$ and $\cU^\varepsilon(y)$ are regions of truncation in simulation-based Bayesian and fiducial inference (see Sections \ref{ss:abc} and \ref{ss:afc}). Finally, let 
  \begin{equation}
    \cG_\theta(y) = \{u\in\Upsilon: G(u, \theta) = y\}
    \label{eq:thetasec}
  \end{equation}
  be the \textit{$\theta$-section} of $\cG(y)$ for each $\theta\in\Theta$.  By definition, $\cG(y) = \{(u\t, \theta\t)\t: \theta\in\Theta, u\in\cG_\theta(y)\}$ and $\cU(y) = \bigcup_{\theta\in\Theta}\cG_\theta(y)$.

The following assumptions are made throughout the rest of the paper.

\begin{assumption}
  Let $\Upsilon$ and $\Theta$ be open subsets of $\real^m$ and $\real^q$, respectively. Assume that
  \begin{longlist}
    \item $G: \Upsilon\times\Theta\to\real^n$ is three-time continuously differentiable;
    \item  the $n\times m$ Jacobian matrix $\nabla_u G(u, \theta)$ has full row rank, and the $n\times q$ Jacobian matrix $\nabla_\theta G(u, \theta)$ has full column rank;
    \item for a given $u$, $\hat\theta(y, u)$ defined by (\ref{eq:thetahat}) is unique.
  \end{longlist}
  \label{as:reg}
\end{assumption}
  \begin{remark}
    Assumption \ref{as:reg} requires that the DGE $G$ is sufficiently smooth in both $u$ and $\theta$, and that the optimal parameter $\hat\theta$ in reproducing the observed $y$ is uniquely identified for each $u$. Both requirements do not apply to parametric models for discrete data \citep[e.g.,][]{Dempster1966, Dempster1968, Stevens1950, Hannig2009}. The focus on continuous data models in the present article bears a resemblance with Fisher's fiducial inference in the early days \citep[e.g.,][]{Fisher1930, Fisher1933}.
    \label{rmk:model}
  \end{remark}

  Immediate consequences of Assumption \ref{as:reg} are that $\cG(y)$ is an $(m + q - n)$-dimensional submanifold of $\Upsilon\times\Theta$, and that $\cG_\theta(y)$ is an $(m - n)$-dimensional submanifold of $\Upsilon$. But more importantly, Assumption \ref{as:reg} implies the isomorphism between $\cG(y)$ and $\cU(y)$, which is summarized as Lemma \ref{lem:iso}. The proof can be found in Appendix B of the supplementary document.

\begin{lemma}
  \label{lem:iso}
  Under Assumption \ref{as:reg}, $\cG(y)\subset\Upsilon\times\Theta\subseteq\real^{m+q}$ is isomorphic to $\cU(y)\subset\Upsilon\subseteq\real^m$, both of which are  twice continuously differentiable submanifolds of dimension $m + q - n$. In particular, $\cU(y)$ can be directly defined as the level set
\begin{align}
  \cU(y) =\ & \{u\in\Upsilon: \overline{\nabla_\theta G}(u, \hat\theta(y, u))\t\cr
  &\ \cdot[G(u, \hat\theta(y, u)) - y] = 0\},
  \label{eq:caluy}
\end{align}
in which $\overline{\nabla_\theta G}(u, \theta)$ is an $n\times(n - q)$ orthogonal complement of $\nabla_\theta G(u, \theta)$ that has orthonormal columns and varies smoothly along $u$ and $\theta$, and 
\begin{equation}
  \cG(y) = \{(u\t, \hat\theta(y, u)\t)\t: u\in\cU(y)\}.
  \label{eq:iso}
\end{equation}
In addition, the intrinsic measures of $\cG(y)$ and $\cU(y)$ satisfy
  \begin{equation}
    \lambda_{\cU(y)}(du) = D(u, \theta)^{-1/2}\lambda_{\cG(y)}(du, d\theta),
    \label{eq:com}
  \end{equation}
  In (\ref{eq:com}),
  \begin{align}
    & D(u, \theta) = \det\bigg(\id_q + \Big[ 
        \nabla_\theta G(u, \theta)\t\cr
        &\ \cdot(\nabla_u G(u, \theta)\nabla_u G(u, \theta)\t)^{-1}\nabla_\theta G(u, \theta)
    \Big]^{-1}\bigg),
    \label{eq:Dterm}
  \end{align}
in which $\id_q$ denotes a $q\times q$ identity matrix. 
\end{lemma}

In the light of Assumption \ref{as:reg} and Lemma \ref{lem:iso}, we highlight three different ways to interpret the dimension of the data generating manifold $\cG(y)$, i.e., $m + q - n$. 
\begin{longlist}
  \item $(m + q) - n$: Most obviously, the data generating manifold $\cG(y)$ is a submanifold of the $(m + q)$-dimensional space $\Upsilon\times\Theta$ that is implicitly defined by the $n$-dimensional constraint $G(u, \theta) - y = 0$. 
  \item $m - (n - q)$: By Lemma \ref{lem:iso}, $\cG(y)$ is isomorphic to its $u$-projection $\cU(y)$, which is a submanifold of the $m$-dimensional space $\Upsilon$ that is implicitly defined by the $(n - q)$-dimensional constraint $\overline{\nabla_\theta G}(u, \hat\theta(y, u))\t(G(u, \hat\theta(y, u)) - y)$. 
  \item  $q + (m - n)$: Each element of $\cG(y)$ is obtained by bundling a $\theta$ from the $q$-dimensional parameter space $\Theta$ with a $u$ from the $\theta$-section $\cG_\theta(y)$, which is an $(m - n)$-dimensional submanifold of $\Upsilon$.
\end{longlist}
ABC typically operates on the joint space $\Upsilon\times\Theta$ and thus naturally adopts the first view. Meanwhile, the second view is aligned with \citeauthor{HannigEtAl2016}'s \citeyearpar{HannigEtAl2016} treatment of GFI on the space $\Upsilon$. As will be elaborated in Section \ref{ss:bayesfid}, the third view links our geometric perspective back to the conventional definitions of Bayesian posteriors and GFDs.

\begin{example}
  \label{ex:basu}
\begin{figure*}[!t]
  \centering
  \includegraphics[width=\textwidth]{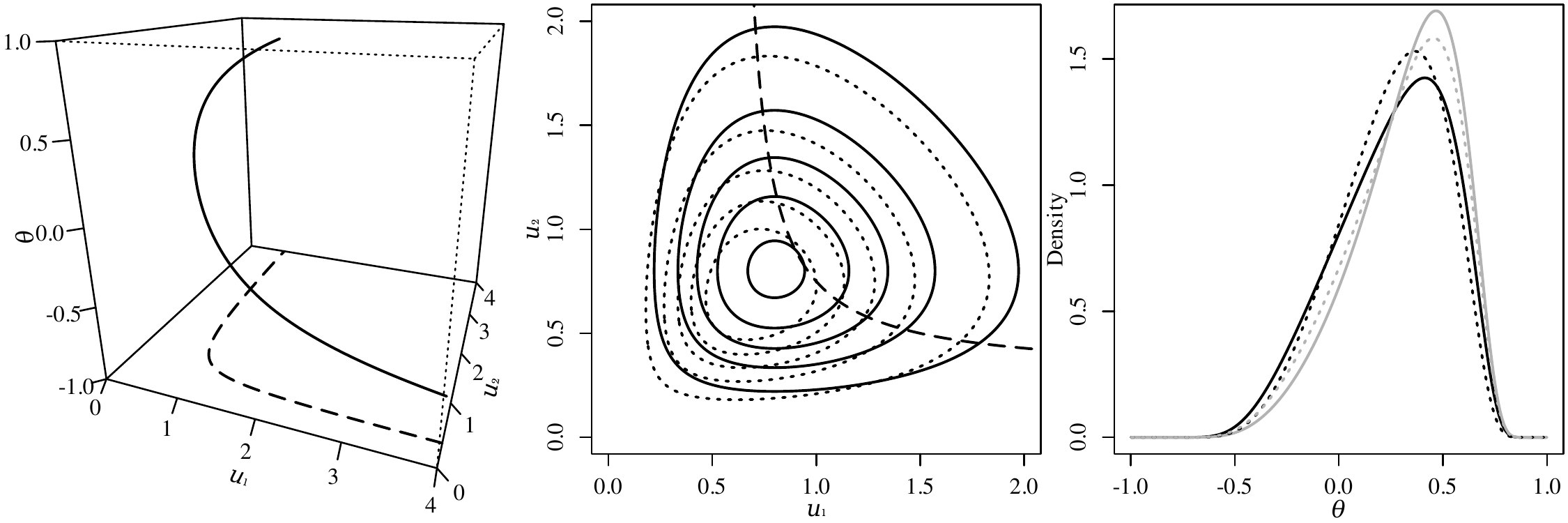}
  \caption{The bivariate Gaussian problem (see Example \ref{ex:basu}). The observed sufficient statistics $y = (1.2, 0.6)\t$, and the sample size $N = 10$. Left: Data generating manifold $\cG(y)$ and its $u$-projection $\cU(y)$. $\cG(y)$ is depicted as solid curves in the three-dimensional space for $(u_1, u_2, \theta)\t$, while $\cU(y)$ is shown as long dashed curves in the two-dimensional subspace for $(u_1, u_2)\t$. Middle: Ambient density contours on the two-dimensional space for $(u_1, u_2)\t$. For generalized fiducial inference (GFI), the ambient density is $\rho(u)$, which is the density of two independent and identically distributed $\chi^2_{N}/N$ random variables and is visualized by the solid contours. Meanwhile, the ambient density for Bayesian inference with the flat prior is given by (\ref{eq:atb1basu}), which is shown as the dotted contours after being normalized as a probability density (by a constant $\approx0.3774$, which is estimated via numerical quadrature). The five contours shown for each method map on to density values 0.1 to 0.9 at an interval of 0.2. The manifold $\cU(y)$ is superimposed as the long dashed curve. Right: $\theta$-marginal densities for GFI (solid) and Bayesian inference (dotted). For GFI, the density obtained from the original data generating equation (DGE) (\ref{eq:basudge}) and the transformed DGE (\ref{eq:basudgea}) are shown in black and gray, respectively. For Bayesian inference, the posterior density corresponding to the flat and Jeffreys priors are shown in black and gray, respectively. Normalizing constants for all three distributions are obtained by numerical quadrature.}
  \label{fig:basu}
\end{figure*}
Suppose that $(X_i, W_i)\t$, $i = 1,\dots, N$, are independent and identically distributed (i.i.d.) bivariate Gaussian random vectors with zero means, unit variances, and a correlation parameter $\theta\in(-1, 1)$. Consider the minimal sufficient statistics $Y_1 = (2N)^{-1}\sum_{i=1}^N(X_i + W_i)^2$ and $Y_2 = (2N)^{-1}\sum_{i=1}^N(X_i - W_i)^2$ for $\theta$. The associated DGE for $Y = (Y_1, Y_2)\t$ is given by
  \begin{equation}
    Y = G(U, \theta) = \left( (1 + \theta)U_1, (1 - \theta)U_2 \right)\t,
    \label{eq:basudge}
  \end{equation}
  in which $U = (U_1, U_2)\t$, and $U_1$ and $U_2$ are i.i.d. $\chi^2_N/N$ variates. In this problem, we have $m = n = 2$ and $q = 1$. Given observed statistics $y = (y_1, y_2)\t$, the data generating manifold $\cG(y)$ is
  \begin{align}
    \cG(y) =\ & \bigg\{(u, \theta)\t\in(0, \infty)^2\times(-1, 1):\cr
    &\ (1 + \theta)u_1 = y_1, (1 - \theta)u_2 = y_2\bigg\},
    \label{eq:Gbasu}
  \end{align}
  in which $u = (u_1, u_2)\t$. To obtain the $u$-projection $\cU(y)$, note that $\nabla_\theta G(u, \theta) = (u_1, -u_2)\t$, and thus its orthogonal complement $\overline{\nabla_\theta G}(u, \theta) = (u_2, u_1)\t/\sqrt{u_1^2 + u_2^2}$. 
    %Furthermore,  
    %\begin{equation*}
    %\hat\theta(y, u) = (-u_1^2 + u_2^2 + u_1y_1 - u_2y_2) / (u_1^2 + u_2^2)
    %\end{equation*}
    %can be uniquely solved from the DGE for fixed $y, u\in(0, \infty)^2$. 
    It follows that the $u$-projection of $\cG(y)$ can be expressed by 
    \begin{equation}
      \cU(y) = \left\{u\in(0,\infty)^2: \frac{2u_1u_2 - u_2y_1 - u_1y_2}{\sqrt{u_1^2 + u_2^2}}= 0\right\}.\kern-12pt
      \label{eq:Ubasu}
    \end{equation}
  The isomorphism between $\cG(y)$ and $\cU(y)$, which are both one-dimensional manifolds, is illustrated in the left panel of Figure \ref{fig:basu}. Finally, the $\theta$-section of $\cG(y)$ is a singleton
  \begin{equation}
    \cG_\theta(y) = \left\{\left(\frac{y_1}{1 + \theta}, \frac{y_2}{1 - \theta}\right)\t\right\},
    \label{eq:Gthetabasu}
  \end{equation}
  which amounts to a zero-dimensional manifold.

\end{example}

\subsection{Bayesian and Fiducial Inference}
\label{ss:bayesfid}

Before presenting our geometric characterization, we establish the following representation of the likelihood function.

\begin{lemma}
    \label{lem:lik}
    The likelihood function $f(y|\theta)$ can be expressed by
  \begin{align}
    f(y|\theta) =\ & \int_{\cG_\theta(y)}\frac{\rho(u)}{\det\left( \nabla_u G(u, \theta)\nabla_u G(u, \theta)\t \right)^{1/2}}\cr
    &\ \cdot\lambda_{\cG_\theta(y)}(du).
    \label{eq:lik}
  \end{align}
\end{lemma}

\begin{proof}
For a fixed $\theta$, $G(\cdot, \theta)$ is a submersion by \romannumeral2) of Assumption \ref{as:reg}; therefore, the level set $\cG_\theta(y)$ is a submanifold of $\Upsilon$ of dimension $m - n$ \citep[][Corollary 5.13]{Lee2013}. (\ref{eq:lik}) follows from the co-area formula (e.g., \citealp{DiaconisEtAl2013}, Theorem 2; \citealp{Federer1996}, Section 3.2.12; \citealp{LelievreEtAl2010}, Lemma 3.2):
  \begin{align}
    &\pr\{Y\in B|\theta\} = \pr\{G(U,\theta)\in B|\theta\}\cr
    =\ & \int_{G(U, \theta)\in B} \rho(u)du \cr
    =\ & \int_B \bigg\{\int_{\cG_\theta(y)}\frac{\rho(u)}{\det\left( \nabla_u G(u, \theta)\nabla_u G(u, \theta)\t \right)^{1/2}}\cr
    &\ \cdot\lambda_{\cG_\theta(y)}(du)  \bigg\} dy
    \label{eq:coarea}
  \end{align}
  for any measurable $B\subseteq\cY$.
\end{proof}

\begin{remark}
  The $\theta$-section of the data generating manifold, i.e., $\cG_\theta(y)$, reduces to the single point $\hat u(y, \theta)$ (see Section \ref{ss:afc}) provided $n = m$ and Assumption \ref{as:reg} holds. (\ref{eq:lik}) then becomes
  \begin{align*}
    & f(y|\theta) =\rho(\hat u (y, \theta))\cr
    &\ \cdot\det\left(\nabla_u G(\hat u (y, \theta), \theta)\nabla_u G(\hat u (y, \theta), \theta)\t\right)^{-1/2}.
  \end{align*}
  Note that \citet{HannigEtAl2016} derived the same likelihood representation under slightly weaker differentiability assumptions (i.e., Assumptions A.1--A.4 on pp. 1--2 of their supplementary document).
  \label{rmk:2016}
\end{remark}

We are now ready to associate Bayesian posteriors and GFDs with limits on $\cG(y)$ and $\cU(y)$. We consider Bayesian inference first in the next proposition; the proof can be found in Appendix C of the supplementary document.

\begin{proposition}
  For Bayesian inference, the general notations of Theorem \ref{thm:gen} reduce to $x = (u\t, \theta\t)\t$, $\cX = \Upsilon\times\Theta$, $\cM^\varepsilon = \cG^\varepsilon(y)$, $\cM = \cG(y)$, $a(x) = \pi(\theta)\rho(u)$, and $h(x) = G(u, \theta) - y$.  Under the assumptions of Theorem \ref{thm:gen}, the weak limit of (\ref{eq:abcjoint}) as $\varepsilon\downarrow 0$ has the following absolutely continuous density 
  \begin{align}
    & f_B(u, \theta) \propto \rho(u)\pi(\theta)\cr
    &\ \cdot\det\Big(\nabla G(u, \theta)\nabla G(u, \theta)\t\Big)^{-1/2}
    \label{eq:limitbayes}
  \end{align}
  with respect to $\lambda_{\cG(y)}$, in which 
  \begin{equation*}
    \nabla G(u, \theta) = \left(\nabla_u G(u, \theta): \nabla_\theta G(u, \theta)\right).
  \end{equation*}
  Equivalently, the limit can be characterized by the density
  \begin{align}
  &\tilde f_B(u;y) \propto \frac{\rho(u)\pi(\hat\theta(y, u))}{\det\left(\nabla_\theta G(u, \hat\theta(y, u))\t\nabla_\theta G(u, \hat\theta(y, u))  \right)^{1/2}}\kern-26pt\cr
&\ \cdot\det\Big(\overline{\nabla_\theta G}(u, \hat\theta(y, u))\t
    \nabla_u G(u, \hat\theta(y, u))\cr
    &\ \cdot\nabla_u G(u, \hat\theta(y, u))\t\overline{\nabla_\theta G}(u, \hat\theta(y, u))\Big)^{-1/2}
    \label{eq:limitbayes1}
  \end{align}
with respect to $\lambda_{\cU(y)}$. Moreover, the density of $\hat\theta(y, u)$ under (\ref{eq:limitbayes1}), or equivalently the $\theta$-marginal of (\ref{eq:limitbayes}), is proportional to $\pi(\theta)f(y|\theta)$.
   \label{prop:bayes}
\end{proposition}

\begin{remark}
  Due to the involvement of a DGE, our geometric characterization of Bayesian inference appears to violate the likelihood principle \cite[e.g.,][Section 1.6.4]{Berger1985}. For instance, the DGE considered in Example \ref{ex:basu} is based on minimal sufficient statistics, which leads to a geometric setup with $m = n = 2$. If we use a DGE corresponding to individual data, such as
  \begin{equation*}
    \begin{bmatrix}
      X_i\\W_i
    \end{bmatrix}
    =\begin{bmatrix}
      1 & 0\\
      \theta & \sqrt{1 - \theta^2}
    \end{bmatrix}\begin{bmatrix}
      U_{i1}\\U_{i2}
    \end{bmatrix},\ i = 1,\dots, N,
  \end{equation*}
  where $U_{i1}, U_{i2}$ are i.i.d. ${\cal N}(0, 1)$ variates, then we have $m = n = 2N$. However, if different DGEs and the corresponding geometric setups yield the same likelihood function via Lemma \ref{lem:lik}, then the last part of Proposition \ref{prop:bayes} guarantees that Bayesian inference made along the $\theta$-marginal should be invariant to choices of DGEs and thus still obeys the likelihood principle.
  \label{eq:likeprin}
\end{remark}

Proposition \ref{prop:gfi} gives a similar characterization for fiducial distributions; the proof can also be found in Appendix C of the supplementary document. 

\begin{proposition}
  \label{prop:gfi}
For GFI, the general notations of Theorem \ref{thm:gen} reduce to $x = u$, $\cX = \Upsilon$, $\cM^\varepsilon = \cU^\varepsilon(y)$, $\cM = \cU(y)$, $a(x) = \rho(u)$, and 
\begin{equation*}
  h(x) = \overline{\nabla_\theta G}(u, \hat\theta(y, u))\t(G(u, \hat\theta(y, u)) - y).
\end{equation*}
Under the assumptions of Theorem \ref{thm:gen}, the weak limit of (\ref{eq:fidu}) as $\varepsilon\downarrow 0$ has the following absolutely continuous density
  \begin{align}
    &\tilde f_F(u;y) \propto \rho(u)\cr
    &\ \cdot\det\Big(\overline{\nabla_\theta G}(u, \hat\theta(y, u))\t\nabla_u G(u, \hat\theta(y, u))\cr
    &\ \cdot\nabla_u G(u, \hat\theta(y, u))\t\overline{\nabla_\theta G}(u, \hat\theta(y, u))\Big)^{-1/2}
    \label{eq:limitfid}
  \end{align}
  with respect to $\lambda_{\cU(y)}$. Equivalently, the limit can be characterized by the density
  \begin{align}
   & f_F(u, \theta) \propto\rho(u)\det\left(\nabla_\theta G(u, \theta)\t\nabla_\theta G(u, \theta)\right)^{1/2}\cr
    &\cdot\det\Big(\nabla G(u, \theta)\nabla G(u, \theta)\t\Big)^{-1/2}
    \label{eq:limitfid1}
  \end{align}
with respect to $\lambda_{\cG(y)}$. Moreover, the density of $\hat\theta(y, u)$ under (\ref{eq:limitfid}), or equivalently the $\theta$-marginal of (\ref{eq:limitfid1}), is proportional to 
  \begin{align}
    &\int_{\cG_\theta(y)}\rho(u)\cdot \det\left( \nabla_\theta G(u, \theta)\t\nabla_\theta G(u, \theta) \right)^{1/2}\cr
    &\ \cdot\det\left(\nabla_u G(u, \theta)\nabla_u G(u, \theta)\t\right)^{-1/2}\lambda_{\cG_\theta(y)}(du).
    \label{eq:margfid}
  \end{align}
\end{proposition}

\begin{remark}
  By the Matrix Determinant Lemma, the second determinant on the right-hand side of (\ref{eq:limitfid1}) can be alternatively expressed as
  \begin{align}
    & \det\left(\nabla_u G(u, \theta)\nabla_u G(u, \theta)\t \right)^{1/2}\cr
    &\ \cdot \det\Big(\id_q + \nabla_\theta G(u, \theta)\t\cr
    &\ \cdot \left[\nabla_u G(u, \theta)\nabla_u G(u, \theta)\t\right]^{-1} \nabla_\theta G(u, \theta)\Big)^{1/2}.
    \label{eq:Jtermdenom}
  \end{align}
  (\ref{eq:Jtermdenom}) is often computationally more efficient for the reasons that $\nabla_u G(u, \theta)\nabla_u G(u, \theta)\t$ can be highly sparse and structured (e.g., block diagonal in the repeated measures ANOVA example; see Section \ref{s:ex}), and that the second determinant is computed with a small $q\times q$ matrix.
  \label{rmk:comp}
\end{remark}

Propositions \ref{prop:bayes} and \ref{prop:gfi} expose crucial disparities between Bayesian inference and GFI, which we now discuss.

\subsubsection*{Question 1: Can We Express a GFD as a Bayesian Posterior?}

The answer is negative in general. By Theorem \ref{thm:gen}, (\ref{eq:limitfid1}) can be thought as restricting the ambient density
  \begin{equation}
    a_F(u, \theta) = \rho(u)\det\left(\nabla_\theta G(u, \theta)\t\nabla_\theta G(u, \theta)\right)^{1/2}
    \label{eq:afid}
  \end{equation}
  to the data generating manifold, whereas Bayesian inference concerns restricting $a_B(u, \theta) = \rho(u)\pi(\theta)$ to the same manifold. The determinant term in (\ref{eq:afid}) may depend on both $u$ and $\theta$ and thus does not reduce to a prior density $\pi(\theta)$ in general. Nevertheless, if the prior is allowed to be data dependent (i.e., extending $\pi(\theta)$ to $\pi(y, \theta)$) and $\cG_\theta(y)$ is singleton (i.e., $m = n$), then a GFD can be interpreted as a posterior. In this special case, $u\in\cG_\theta(y)$ if and only if $u = \hat u(y, \theta)$, and the dependency on $u$ can therefore be removed from the determinant term in (\ref{eq:afid}) (see Remark \ref{rmk:2016}). The conclusion that GFD is typically not a Bayesian posterior can also be deduced from the $\theta$-marginal density (\ref{eq:margfid}), which cannot be factorized into the product of the likelihood function (\ref{eq:lik}) and a function of $\theta$ in general. In the special case when $\cG_\theta(y)$ is singleton, (\ref{eq:margfid}) reduces to the formula presented in Theorem 1 of \cite{HannigEtAl2016} and is subject to a data-dependent Bayesian interpretation.

\subsubsection*{Question 2: Is There a Unique Justification for the Objectivity of GFI?}

A new qualification of objective post-data inference is manifested by the limiting density on $\cU(y)$. Contrasting (\ref{eq:limitbayes1}) with (\ref{eq:limitfid}), we observe that the ambient densities (on $\Upsilon$, in the sense of Theorem \ref{thm:gen}) for Bayesian inference and GFI are
\begin{equation}
  \tilde a_B(u;y) = \frac{\rho(u)\pi(\hat\theta(y, u))}{\det\left(\nabla_\theta G(u, \hat\theta(y, u))\t\nabla_\theta G(u, \hat\theta(y, u))  \right)^{1/2}}\kern-26pt
  \label{eq:atbayes}
\end{equation}
and $\tilde a_F(u) = \rho(u)$, respectively. Hence, GFI is objective in the sense that it ``continue[s] to regard'' \citep[][p. 885]{Dempster1963} the random components $U$'s distribution as the ambient distribution on $\Upsilon$. This feature leaves the data generation process intact and demands no extra information. To the contrary, the ambient density (\ref{eq:atbayes}) corresponding to Bayesian inference is adapted using prior information encoded in $\pi$ as well as the observed data $y$.

\subsubsection*{Question 3: Are Bayesian Inference and GFI Invariant to Transformations of the Data Space?}
It is known that Bayesian inference is invariant to nonlinear transformations of the DGE whereas GFI is generally not. Consider a differentiable transform $\varphi:\real^n\to\real^n$ such that the $n\times n$ Jacobian matrix $\nabla\varphi$ always has full rank. Compose $\varphi$ with $G$ to form a transformed DGE, i.e., $\varphi\circ G$. Because $G(u, \theta) = y$ for all $(u\t, \theta\t)\t\in\cG(y)$, we have
\begin{align*}
  &\det\left(\nabla(\varphi\circ G)(u, \theta)\nabla(\varphi\circ G)(u, \theta)\t\right)\cr
  =\ & \det\left(\nabla\varphi(y)\right)^2\det\left(\nabla G(u, \theta)\nabla G(u, \theta)\t\right),
\end{align*}
which is proportional to $\det\left(\nabla G(u, \theta)\nabla G(u, \theta)\t\right)$. Therefore, (\ref{eq:limitbayes}) for Bayesian inference remains unchanged when the transformed DGE $\varphi\circ G$ is used in place of $G$. Meanwhile, 
\begin{align}
  &\det\left(\nabla_\theta(\varphi\circ G)(u, \theta)\t\nabla_\theta(\varphi\circ G)(u, \theta)\right)\cr
  =\ & \det\left(\nabla_\theta G(u, \theta)\t\nabla\varphi(y)\t\nabla\varphi(y)\nabla_\theta G(u, \theta)\right).
  \label{eq:detfid}
\end{align}
The right-hand side of (\ref{eq:detfid}) is not proportional to\break $\det\left( \nabla_\theta G(u, \theta)\t \nabla_\theta G(u, \theta)\right)$ in general. There are two notable exceptions: when $n = q$ so that $\nabla_\theta G(u, \theta)$ is square and when $\nabla\varphi(y)\t\nabla\varphi(y)\equiv c\,\id_n$ with a positive constant $c$. Consequently, (\ref{eq:limitfid1}) for GFI is typically not invariant under the transformed DGE. Applying (\ref*{eq:detnabh}) from the supplementary document, the same conclusion can be drawn for the limiting densities on $\cU(y)$, i.e., (\ref{eq:limitbayes1}) and (\ref{eq:limitfid}).

\addtocounter{example}{-1}
\begin{example}[continued]
  Using (\ref{eq:basudge}) as the DGE for the bivariate Gaussian problem, we can express (\ref{eq:afid}), i.e., GFI's ambient density on $\cG(y)$, as
  \begin{equation}
    a_F(u, \theta)\equiv a_F(u) = \rho(u)\sqrt{u_1^2 + u_2^2},
    \label{eq:afidbasu}
  \end{equation}
  in which $\rho$ is the joint density of two i.i.d. $\chi^2_N/N$ variates. Note that $m = n = 2$ in this example, and that $\hat u(y, \theta) = (y_1/(1 + \theta), y_2/(1 - \theta) )\t$. Therefore, the GFD coincides with the Bayesian posterior resulted from the data-dependent prior
    \begin{equation}
      \pi(y, \theta) \propto\sqrt{\frac{y_1^2}{(1 + \theta)^2} + \frac{y_2^2}{(1 - \theta)^2}}
    \label{eq:ddpbasu}
    \end{equation}

    Next, consider two prior distributions: the flat prior $\pi^{(1)}(\theta)\equiv 1/2$ and the Jeffreys prior $\pi^{(2)}(\theta)= \sqrt{1+\theta^2}/(1 - \theta^2)$. On the one hand, (\ref{eq:atbayes}) for the flat prior is simplified to
  \begin{equation}
    \tilde a_B^{(1)}(u; y)\equiv\tilde a_B^{(1)}(u) = \frac{\rho(u)}{2\sqrt{u_1^2 + u_2^2}},
    \label{eq:atb1basu}
  \end{equation}
  which does not depend on $y$. Whenever $N > 2$, (\ref{eq:atb1basu}) is bounded and integrable; it can thus be normalized to a proper density on $(0, \infty)^2$. The contours of $\tilde a_B^{(1)}(u)$ and $\tilde a_F(u) = \rho(u)$ are contrasted in the middle panel of Figure \ref{fig:basu}.  On the other hand, (\ref{eq:atbayes}) for the Jeffreys prior becomes 
  \begin{align}
    & \tilde a_B^{(2)}(u; y) = \rho(u)\sqrt{\frac{(u_1^2 + u_2^2)}{2}}\bigg[ \frac{1}{(2u_2^2 + u_1y_1 - u_2y_2)^2}\kern-16pt\cr
    &\ + \frac{1}{(2u_1^2 - u_1y_1 + u_2y_2)^2}\bigg]^{1/2}.
    \label{eq:atb2basu}
  \end{align}
  Unlike $\tilde a_F(u)$ and $\tilde a_B^{(1)}(u)$, (\ref{eq:atb2basu}) is data dependent and unbounded on $(0, \infty)^2$ for all $N$: It diverges to infinity as $2u_2^2 + u_1y_1 - u_2y_2$ or $2u_1^2 - u_1y_1 + u_2y_2$ approaches zero.

  Finally, consider the reciprocal transformation $\varphi:(0, \infty)^2\to(0, \infty)^2$, $\varphi(y)\mapsto (1/y_1, 1/y_2)\t$. Using the transformed DGE
  \begin{equation}
    (\varphi\circ G)(u, \theta) = \left( \frac{1}{(1 + \theta)u_1}, \frac{1}{(1 - \theta)u_2} \right)\t
    \label{eq:basudgea}
  \end{equation}
for GFI, we arrive at an ambient density on $\cG(y)$ that is different from (\ref{eq:afidbasu}):
  \begin{equation}
    a_F^{(\varphi)}(u, \theta) = \rho(u)\sqrt{\frac{1}{u_1^2(1 + \theta)^4} + \frac{1}{u_2^2(1 - \theta)^4}}.
    \label{eq:afidbasua}
  \end{equation}
  The data-dependent prior derived based on $\varphi\circ G$ is also different from (\ref{eq:ddpbasu}):
    \begin{equation}
      \pi^{(\varphi)}(y, \theta) \propto\sqrt{\frac{1}{y_1^2(1 + \theta)^2} + \frac{1}{y_2^2(1 - \theta)^2}}.
    \label{eq:ddpbasua}
    \end{equation}
    The corresponding $\theta$-marginals of GFDs (using the original $G$ versus the transformed $\varphi\circ G$) and Bayesian posteriors (using the flat prior versus the Jeffreys prior) are contrasted in the right panel of Figure \ref{fig:basu}.
  
\end{example}

\section{Review of Markov Chain Monte Carlo Sampling on Manifolds}
\label{s:mcmc}
Monte Carlo approximations to a fiducial or a Bayesian posterior distribution---when viewed as an absolutely continuous distribution defined on a smooth manifold---can be constructed via manifold MCMC sampling. In this section, we review a manifold random-walk Metropolis (RWM) algorithm proposed by \citet{ZappaEtAl2018}. We focus on a specific Gaussian proposal that corresponds to a one-step discretization of the constrained overdamped Langevin process \citep[][Section 3.3]{LelievreEtAl2012}. For generality, we adopt the notation of Theorem \ref{thm:gen} in the current section. The algorithm is presented assuming that $\cM$ is unbounded, though incorporating additional inequality constraints is straightforward \citep[see, e.g., Remark 6 of][]{LelievreEtAl2019}.

\begin{algorithm}[!t]
  \caption{Manifold Random-Walk Metropolis Update}
  \label{alg:rwm}
  \begin{algorithmic}[1]
    \Require Initial value $x\in\cM$, target density $f:\real^d\to\real$, proposal parameters $\mu(x)\in\real^{d - n}$, $\Sigma(x)\in\real^{(d - n)\times(d - n)}_+$, and $\delta > 0$, tuning parameters $\gamma$ and $R$ for {\tt Project} (Algorithm \ref{alg:ret})
    \State Sample $z$ from ${\cal N}(\mu(x), \delta\Sigma(x))$ and set $w = \overline{\nabla h}(x)z$\label{ln:tang}
    \State Propose $x' = \hbox{\tt Project}(x + w, \nabla h(x), \gamma, R)$\label{ln:proj}
    \If{fail to find $x'$}\label{ln:check1a}
      \State\Return $x$
      \EndIf \label{ln:check1b}
      \State Compute $z' = \overline{\nabla h}(x')\t(x - x')$ and $w' = \overline{\nabla h}(x')z'$ \label{ln:check2a}
      \State Set $x'' = \hbox{\tt Project}(x' + w', \nabla h(x'), \gamma, R)$\label{ln:rproj}
    \If{fail to find $x''$ or $x''\neq x$}
      \State\Return $x$
    \EndIf\label{ln:check2b}
    \State Compute\label{ln:check3a} 
    \begin{equation}
      \alpha(x; x') = \min\left\{1, \frac{f(x')\phi(z'; \mu(x'), \delta\Sigma(x'))}{f(x)\phi(z; \mu(x), \delta\Sigma(x))}\right\}
      \label{eq:mh}
    \end{equation}
    where $\phi(\cdot; \mu, \Sigma)$ denotes the density of ${\cal N}(\mu, \Sigma)$
    \State Sample $u\in(0, 1)$ from $\hbox{Unif}(0, 1)$
    \If{$u\le\alpha(x; x')$}
      \State\Return $x'$
      \Else
      \State\Return $x$
    \EndIf\label{ln:check3b}
  \end{algorithmic}
\end{algorithm}

\begin{algorithm}[!t]
  \caption{\texttt{Project} to Manifold $\cM$ along $B$}
  \label{alg:ret}
  \begin{algorithmic}[1]
    \Require Initial location $x_0\in\real^d$, full-rank basis matrix $B\in\real_+^{d\times n}$, convergence tolerance $\gamma > 0$, maximum number of iterations $R$
    \State Set $a_0 = 0_n$, where $0_n$ is a $n\times 1$ vector of zeros
    \For{$r = 0, \dots, R - 1$}
      \If{$\|h(x_r)\| \le \gamma$}
        \State \Return $x_r$
      \EndIf
      \State Update $a_{r + 1} = a_r - \left[ \nabla h(x_r + Ba_r)\t B \right]^{-1}h(x + Ba_r)$
      \State Compute $x_{r + 1} = x_0 + Ba_r$
    \EndFor
    \State Throw an error
  \end{algorithmic}
\end{algorithm}

\subsection{Manifold Random-Walk Metropolis}
\label{ss:rwm}
The pseudocode for a single manifold RWM update is summarized in Algorithm \ref{alg:rwm}. With a slight abuse of notation, $f$ in the pseudocode denotes a smooth extension of the target density $f$ (with respect to $\lambda_\cM$) to the ambient space $\real^d$. Given an initial value $x$ on the manifold $\cM$, a proposal $x'$ is generated from a random walk on the tangent space at $x$ (Line \ref{ln:tang}), followed by a projection back to the manifold along the normal direction (Line \ref{ln:proj}).\footnote{We follow \citet[][p. 2610]{ZappaEtAl2018} to call the operation a ``projection''; however, it is different from an orthogonal projection to the manifold.} Let $T_x\cM = \{w\in\real^d: \nabla h(x)w = 0\}$ be the tangent space of $\cM$ at $x$, and $\overline{\nabla h}(x)\in\real^{d\times(d - n)}$ be an orthogonal complement of $\nabla h(x)\t\in\real^{d\times n}$ with orthonormal columns. $\overline{\nabla h}$ forms a basis for $T_x\cM$. The random-walk step entails generating $w = \overline{\nabla h}(x)z\in T_x\cM$, in which $z$ follows ${\cal N}(\mu(x), \delta\Sigma(x))$, $\mu(x)\in\real^{d - n}$, $\Sigma(x)\in\real^{(d - n)\times(d - n)}_+$ is positive definite, and $\delta > 0$ is the proposal scale parameter. The point $x + w$ resulted from the random walk needs to be retracted back to $\cM$ to yield a valid proposal. In particular, we find a coefficient vector $a\in\real^n$ that solves $h(x + z + \nabla h(x)a) = 0$. Because the constraint function $h$ is generally nonlinear, we follow \citet{ZappaEtAl2018} to apply a standard Newton solver. This projection step is abbreviated as \texttt{Project} in the pseudocode: The four arguments required by the function call of \texttt{Project} are described in the input line of Algorithm \ref{alg:ret}. 

The proposal $x'$ is not accepted unless it passes all the following three checks. First, it is possible that the function \texttt{Project} throws an error---or equivalently, the Newton solver fails to converge (see Lines \ref{ln:check1a}--\ref{ln:check1b} of Algorithm \ref{alg:rwm}); if so, we have to revert to the original $x$ and proceed to the next cycle. Second, we need to confirm that a reverse move starting from $x'$ recovers the original point $x$ (Lines \ref{ln:check2a}--\ref{ln:check2b}); a graphical illustration for the potential failure of such a reversal move can be found in Figure 2 of \citet{LelievreEtAl2019}. Finally, a standard Metropolis-Hastings step is performed (Lines \ref{ln:check3a}--\ref{ln:check3b}), in which the acceptance ratio is given by (\ref{eq:mh}). It was shown in \citet{ZappaEtAl2018} that the above RWM update satisfies the detailed balance condition when the equations were solved exactly in the retraction steps (Lines \ref{ln:proj} and \ref{ln:rproj}). When a numerical solver is employed, which is typically the case in practice, the manifold RWM algorithm can be understood as a noisy MCMC method \citep{AlquierEtAl2016}.

\subsection{Proposal Distribution}
\label{ss:prop}
We found in pilot experiments that a Gaussian proposal (Line \ref{ln:tang}) with
\begin{equation}
  \mu(x)\t = \frac{\delta^2}{2}\nabla\log f(x)\overline{\nabla h}(x)
  \label{eq:mu}
\end{equation}
and $\Sigma(x)\equiv \id_{(d - n)\times (d - n)}$ fares efficient even when the dimension of the manifold (i.e., $d - n$) is high. The corresponding manifold RWM update yields an Euler discretization of the constrained overdamped Langevin diffusion \citep[with an identity mass matrix;][Proposition 3.6]{LelievreEtAl2012}, which is also equivalent to a single update of ``position'' \citep[][p. 383]{LelievreEtAl2019} while simulating the constrained Hamilton dynamics via the RATTLE discretization.\footnote{One can extend Algorithm \ref{alg:rwm} to a manifold Hamiltonian Monte Carlo sampler by repeatedly executing Lines \ref{ln:tang}--\ref{ln:check2b} with a small, fixed ``timestep'' $\delta$ \citep[][p. 380]{LelievreEtAl2019}.} In case the gradient of the log density is challenging to evaluate (e.g., the gradient of the log fiducial density (\ref{eq:limitfid1}) involves the second derivatives of the DGE), we may substitute $\nabla\log f(x)$ in (\ref{eq:mu}) by a numerical estimate. In fact,
\begin{align}
  & \nabla\log f(x)\overline{\nabla h}(x)\cr
  =\ & \nabla_t\log f(x + \overline{\nabla h}(x)\t t)\big|_{t = 0_{n - d}}.
  \label{eq:derivtrick}
\end{align}
Compared to differentiating $\log f$ with respect to $x\in\real^d$, the right-hand side derivative in (\ref{eq:derivtrick}) is taken with respect to the lower-dimensional $t\in\real^{d - n}$ and thus can be more economical to numerically approximate.

\section{Example: GFI for Repeated Measures ANOVA}
\label{s:ex}

Next, we apply our main result and sampling strategy to perform GFI in a repeated measures ANOVA example. GFI for Gaussian linear mixed-effects models has been studied by \citet{CisewskiHannig2012}; however, their development is confined to the $\varepsilon$-fatting, i.e., (\ref{eq:fidu}), with a positive tolerance $\varepsilon$, and the proposed sequential Monte Carlo algorithm suffers from numerical degeneracy when $\varepsilon$ is small. From the new geometric perspective, we can not only express the exact fiducial density (i.e., the weak limit as $\varepsilon\downarrow 0$) but also generate fiducial samples conveniently using manifold MCMC algorithms. Meanwhile, Bayesian inference for linear mixed-effects models has been extensively studied and widely applied for decades \citep[e.g.,][Chapter 15, and the bibliographic note therein]{GelmanEtAl2013}. As a Bayesian benchmark for the empirical data example (Section \ref{ss:emp}), we consider a weakly informative prior configuration suggested by \citet{Gelman2006} and approximate the resulting posterior via Gibbs sampling with data augmentation \citep{TannerWong1987}.

\subsection{Model}
\label{ss:fa}
In a within-subject design, let $X_{ij}$ denote the observed response of subject $j$ in condition $i$, where $i = 1,\dots, I$ and $j = 1, \dots, J$ with $I, J > 1$. Repeated measures ANOVA decomposes each response entry $X_{ij}$ into the sum of the treatment mean $\mu_i$, subject effect $\sigma_z Z_j$, and the interaction effect $\sigma_e E_{ij}$:
\begin{equation}
  X_{ij} = \mu_i + \sigma_z Z_j + \sigma_e E_{ij},
  \label{eq:rmanova}
\end{equation}
in which $Z_j$ and $E_{ij}$ are continuous random variables with known distributions, and $\sigma_z, \sigma_e > 0$ are the respective scale parameters. (\ref{eq:rmanova}) amounts to the component-wise expression of the DGE. To be consistent with our generic notation, identify 
\begin{align}
  &Y = \vec(X),\
  U = (Z\t, \vec(E)\t)\t,\cr
  &\theta = (\mu\t, \sigma_z, \sigma_e)\t,
  \label{eq:elem}
\end{align}
in which $X = \{X_{ij}\}$, $Z=\{Z_j\}$, $E=\{E_{ij}\}$, and $\mu = \{\mu_i\}$. The dimensions of $Y$, $U$, and $\theta$ are $n = IJ$, $m = IJ + J$, and $q = I + 2$, respectively.

In the next proposition, we verify the crucial tightness assumption that allows us to apply Proposition \ref{prop:gfi}. The proof can be found in Appendix D in the supplementary document.

\begin{proposition}
  Under a repeated measures ANOVA model, suppose that $\vec(E)$ and $Z$ are independent, spherically distributed random vectors on $\real^{IJ}$ and $\real^J$, respectively. Then the collection of probability measures (\ref{eq:fidu}) indexed by $\varepsilon$ is tight.
 \label{prop:rmanova}
\end{proposition}

\begin{remark}
  The additional distributional assumption for $\vec(E)$ and $Z$ is made for ease of theoretical justification. A spherically distributed variate $S$ is subject to a unique factorization $S = RV$, where $V$ is uniform on the unit sphere and $R$ is a positive, continuous random variable \citep{FangEtAl1990}. Common examples of spherical distributions are multivariate Gaussian and $t$ distributions, which are popular choices of error distributions for linear models \citep{FraserNg1980}. 
  \label{rmk:sph}
\end{remark}

\begin{remark}
  The proof of Proposition \ref{prop:rmanova} in Appendix D in the supplementary document can be adapted to handle unbalanced designs (i.e., $i = 1, \dots, I_j$ where $I_j$'s may not be identical for different $j$'s) or even more general linear mixed-effects models considered by \citet{CisewskiHannig2012}. We only need to modify the definition of $\beta$ that appears in (\ref*{eq:rmanova1}) from the supplementary document to include all fixed effects and scale parameters for random effects, and correspondingly the definition of $W(z)$.
  \label{rmk:glmm}
\end{remark}

Pointwise evaluation of the fiducial density (\ref{eq:limitfid1}) requires formulas for the Jacobian matrices. Under the repeated measures ANOVA model, $\nabla_u G(u, \theta)$ and $\nabla_\theta G(u, \theta)$ have blocked matrix representations corresponding to the partitions of $U$ and $\theta$ in (\ref{eq:elem}):
\begin{align}
  &\nabla_u G(u, \theta) = (\id_J\otimes\sigma_z 1_I\ :\ 
  \sigma_e\,\id_{IJ}),\cr
  &\nabla_\theta G(u, \theta) = (1_J\otimes \id_I\ :\ Z\otimes 1_I\
  :\ \vec(E)).
  \label{eq:dGdx}
\end{align}
Note that the dimensions of $\nabla_u G$ and $\nabla_\theta G$ are $IJ\times (IJ + J)$ and $IJ\times (I + 2)$. Na\"ively evaluating the fiducial density and applying the manifold MCMC update incur matrix operations up to $O(I^3J^3)$ complexity, which can be prohibitively expensive when $I$ or $J$ is large. In Appendix E of the supplementary document, we demonstrate that the computational cost can be reduced to $O(I^3J)$ thanks to the specific structure of (\ref{eq:dGdx}).

\begin{figure*}[!t]
  \centering
  \includegraphics[width=0.9\hsize]{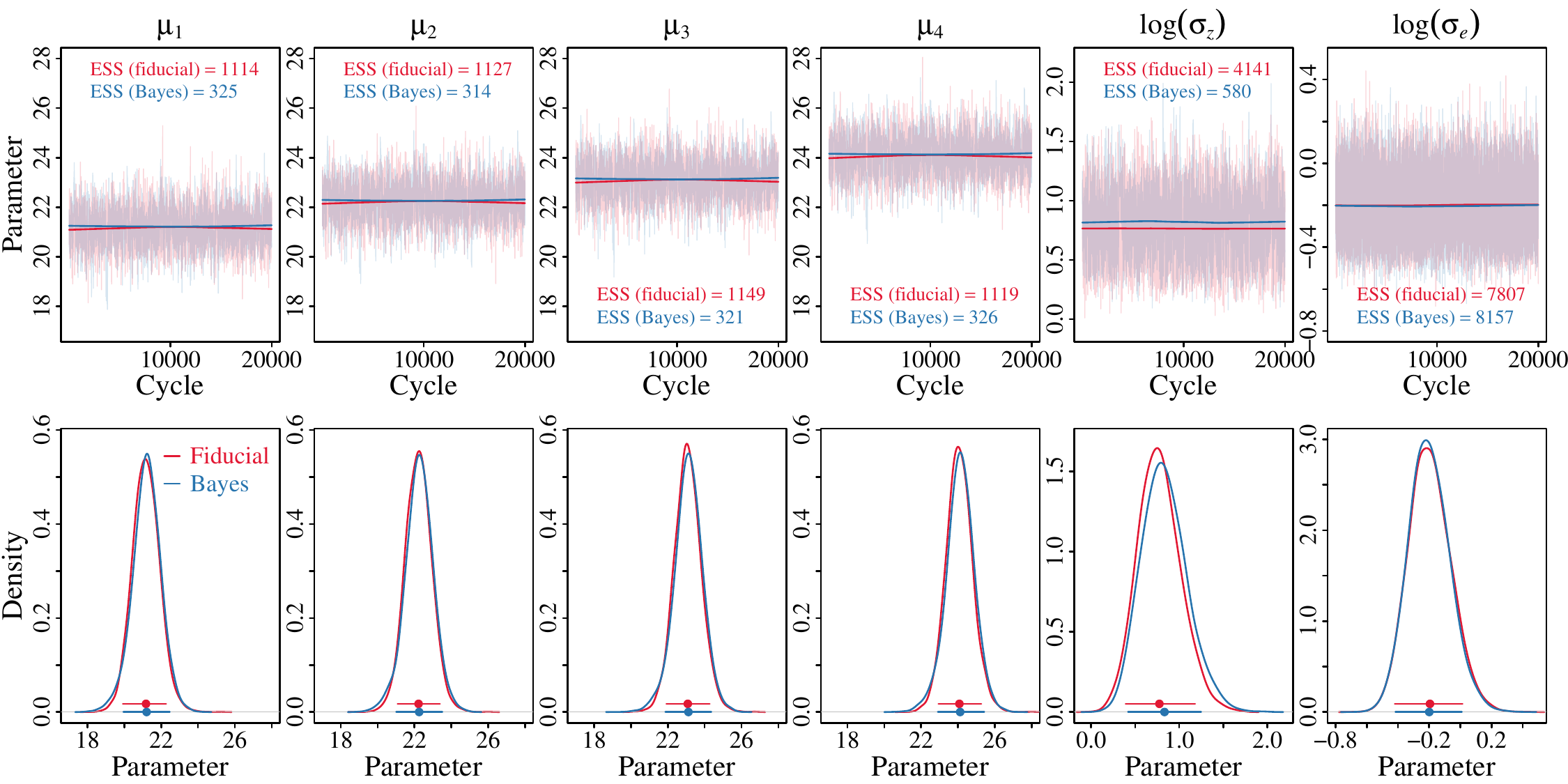}
  \caption{Trace plots for Markov chain Monte Carlo samplers (top panels) and estimated (parameter-by-parameter marginal) densities of the Bayesian posterior and the generalized fiducial distribution (bottom panels) in the orthodontic growth example. Each column represents a parameter in the model. Fiducial and Bayesian results are shown in red and blue colors, respectively. Top: The trace plots are produced from 20000 cycles after burning in the first 10000 to remove the impact of arbitrary starting values. To better visualize trends of sample paths, locally-weighted polynomial regression estimates are superimposed as solid curves. The effective sample size (ESS) statistic is presented for each marginal sample path. Bottom: Estimated fiducial/posterior means (dots) and 90\% highest-density interval estimates (lines) are added to the bottom of density plots.}
  \label{fig:param}
\end{figure*}

\subsection{Empirical Data: Orthodontic Growth}
\label{ss:emp}
Using the orthodontic growth data \citep{PotthoffRoy1964}, we apply the manifold RWM algorithm to sample from the fiducial limiting density (\ref{eq:limitfid1}). The data set contains measures of the distance between the pituitary and pterygomaxillary fissures for a total number of 27 children, including 16 males and 11 females. Measures were obtained every two years from age 8 to 14, resulting in four measures per child. Only the female subsample ($I = 4$ and $J = 11$) was considered in the our illustration. The distributions of random components are set to $Z\sim{\cal N}(0_J, \id_J)$ and $\vec(E)\sim {\cal N}(0_{IJ}, \id_{IJ})$. For comparison purposes, we also conduct Bayesian inference with a weakly informative prior per the recommendation of \citet{Gelman2006}. Specifically, improper uniform priors were specified for $\mu_1, \dots, \mu_4$ and $\log\sigma_e$; a half-Cauchy prior with scale 34.5 was used for $\sigma_z$, in which the Cauchy scale was set to three times the range of the data \citep[see, e.g.,][]{CisewskiHannig2012}.

To sample from (\ref{eq:limitfid1}) using Algorithm \ref{alg:rwm}, the proposal scale of the Gaussian random walk was set to $\delta = 1.05$. After discarding the first 10000 cycles to remove the influence of an arbitrary starting state, we obtain an empirical acceptance rates of 0.4882 out of 20000 retained MCMC cycles. The tolerance and the maximum iterations of the Newton solver were set to $10^{-6}$ and 50, respectively. Using the same numbers of burn-in and retained cycles, a slice-within-Gibbs sampler was employed to simulate from the augmented posterior distribution of $\theta$ and $Z$. Both sampling algorithms were implemented in \citet{matlab}, and the source code is available upon request.

Trace plots of the generated Markov chains were displayed in the first row of Figure \ref{fig:param}; note that we plot the logarithms of the scale parameters $\sigma_z$ and $\sigma_e$. All the twelve reported univariate sample paths appear to be stationary. The effective sample size \citep[ESS;][Chapter 11]{GelmanEtAl2013} statistics (upon rounding to integers) range from 1114 to 7807 for fiducial samples generated by the manifold sampler, and from 314 to 8157 for posterior samples generated by the slice-within-Gibbs sampler. For both MCMC samplers, lower ESS statistics were observed for the treatment mean parameters. Except for $\log\sigma_e$, i.e., the log-scale parameter for the interaction effect, the manifold RWM sampler yields higher ESS than the slice-within-Gibbs sampler, indicating better sampling efficiency.

We compare marginal GFDs and posteriors for each of the six parameters in the second row of Figure \ref{fig:param}. For treatment mean parameters $\mu_1, \dots, \mu_4$ and the interaction log-scale parameter $\log\sigma_e$, the two sets of distributions are almost identical. In the meantime, the posterior for the subject log-scale parameter $\log\sigma_z$ appears to concentrate on slightly higher values compared to the corresponding GFDs. The same pattern can be identified by contrasting the 90\% highest density interval estimators.
 
\section{Concluding Remarks}
\label{s:end}
In the present paper, we approach Bayesian inference and GFI from a differential geometric perspective. Conditional on the observed data, a statistical model with a smooth DGE (meeting Assumption \ref{as:reg}) defines a submanifold within the joint space of random components  $u$ and parameters $\theta$---namely, the data generating manifold. A Bayesian posterior or a GFD corresponds to the $\theta$-marginal for a joint distribution of $u$ and $\theta$ that is supported on the data generating manifold and has an absolutely continuous density with respect to the intrinsic measure of the manifold. Moreover, the data generating manifold can be equivalently represented by its projection on the space of random components $u$. We also demonstrate that manifold MCMC samplers can be utilized to construct Monte Carlo approximations to the GFD in an empirical example.

Taking an alternative, yet still differential geometric, perspective on GFI, \citet{MurphEtAl2022} defined a GFD whose parameter space $\Theta$ itself is a manifold. In contrast to this paper, where the GFD is defined as the limiting measure of an ambient distribution constrained to a sequence of shrinking sublevel sets (an \textit{extrinsic} perspective), \citet{MurphEtAl2022} define their distribution directly on the manifold (an \textit{intrinsic} perspective) using the smooth local structure. Whenever the dimension of the random component and the observed data are the same ($m = n$), Theorem 3.1 from \citet{Hwang1980} can be used to calculate an extrinsic analogue of the GFD from \citet{MurphEtAl2022}. Under some regularity conditions, \citet{MurphEtAl2022} showed that these extrinsic and intrinsic perspectives converge in the local limit. A natural extension of the main result of this paper is to extend Proposition \ref{prop:gfi} to additionally handle a constrained parameter space, which can be seen as both a generalization of the result from \citet{MurphEtAl2022} (for $m > n$), and as an alternative, extrinsic perspective using limiting measures.

%%%%%%%%%%%%%%%%%%%%%%%%%%%%%%%%%%%%%%%%%%%%%%
%% Supplementary Material, including data   %%
%% sets and code, should be provided in     %%
%% {supplement} environment with title      %%
%% and short description. It cannot be      %%
%% available exclusively as external link.  %%
%% All Supplementary Material must be       %%
%% available to the reader on Project       %%
%% Euclid with the published article.       %%
%%%%%%%%%%%%%%%%%%%%%%%%%%%%%%%%%%%%%%%%%%%%%%
\begin{supplement}
\stitle{Supplementary Document for ``A Geometric Perspective on Bayesian and Generalized Fiducial Inference''}
\sdescription{The supplementary document contains proofs of the theoretical results and additional computational details for the repeated ANOVA example (Section \ref{s:ex}).}
\end{supplement}

% bibliography
\setcounter{secnumdepth}{0}
\bibliographystyle{imsart-nameyear}
\bibliography{FiducialGeometry}

% supplemental file
\clearpage


\section*{Supplementary material}

# Supplementary Document for
# "A Geometric Perspective on Bayesian and Generalized Fiducial Inference"

Yang Liu[1]   Jan Hannig[2]   Alexander C. Murph[2]

# Contents

[1] Department of Human Development and Quantitative Methodology, University of Maryland, College Park, Maryland, USA. Correspondence author. Email: yliu87@umd.edu

[2] Department of Statistics and Operations Research, the University of North Carolina at Chapel Hill, North Carolina, USA.

# Appendix A
# Proof of Theorem 1

The proof proceeds in two stages. We first establish the limit when restricting the sequence of probability measures on a compact tubular neighborhood of the manifold $\mathcal{M}$ (Section A.1). The result is then extended to the entire ambient space $\mathcal{X}$ using the tightness assumption (Section A.2).

## A.1 On a Compact Tubular Neighborhood

Let $N_x\mathcal{M}$ be the normal space of $\mathcal{M}$ at $x \in \mathcal{M}$, and $N\mathcal{M} = \{(x,v) : x \in \mathcal{M}, v \in N_x\mathcal{M}\}$ be the normal bundle of $\mathcal{M}$. By the Tubular Neighborhood Theorem (Lee; 2013, Theorem 10.19), there exists a positive continuous function $\tau : \mathcal{M} \to (0, \infty)$ and a tubular neighborhood of $\mathcal{M}$ that is defined as the diffeomorphic image of the open set $V = \{(x,v) \in N\mathcal{M} : \|v\| < \tau(x)\}$ under the map $\chi : \mathcal{R}^d \times \mathcal{R}^d \to \mathcal{R}^d$, $\chi(x,v) \mapsto x + v$. We denote such a tubular neighborhood by $\mathcal{T} = \chi(V)$ and its closure by $\overline{\mathcal{T}}$.

Now intersect $\mathcal{M}$ with a compact set $K \subset \mathcal{X}$, resulting in $\mathcal{M}_K = \mathcal{M} \cap K$. Let the tubular neighborhood of $\mathcal{M}_K$ be $\mathcal{T}_K = \chi(V_K)$ with the closure $\overline{\mathcal{T}}_K$, where $V_K = \{(x,v) \in N\mathcal{M}_K : \|v\| < \tau(x)\}$, and $N\mathcal{M}_K = \{(x,v) : x \in \mathcal{M}_K, v \in N_x\mathcal{M}\}$ is the normal bundle of $\mathcal{M}_K$. Also define $\mathcal{M}_K^\varepsilon = \mathcal{M}^\varepsilon \cap \overline{\mathcal{T}}_K = \{x \in \overline{\mathcal{T}}_K : \|h(x)\| \leq \varepsilon\}$ be the $\varepsilon$-fattening of $\mathcal{M}_K$ within $\overline{\mathcal{T}}_K$. Without loss of generality, suppose that Assumption 1 ii) is satisfied with $\mathcal{M}_K$ and $\mathcal{M}_K^\varepsilon$ replacing $\mathcal{M}$ and $\mathcal{M}^\varepsilon$. Our goal is to show that, for all bounded continuous function $g : \mathcal{X} \to \mathcal{R}$,

$$\varepsilon^{-n} \int_{\mathcal{M}_K^\varepsilon} g(x)a(x)dx \to \frac{\pi^{n/2}}{\Gamma(n+1/2)} \int_{\mathcal{M}_K} g(x)a(x) \det\left(\nabla h(x)\nabla h(x)^\top\right)^{-1/2} \lambda_\mathcal{M}(dx) \quad (S.1)$$

as $\varepsilon \downarrow 0$. It follows from (S.1) and the Portmanteau Lemma that the probability measures $\{P_\varepsilon|_{\overline{\mathcal{T}}_K}\}$ converges weakly to $P_0|_{\mathcal{M}_K}$, in which $P|_A$ stands for the restriction of measure $P$ on $A$.

By definition, any $x \in \overline{\mathcal{T}}_K$ is subject to a unique decomposition $x = s + Q(s)t$ where $s \in \mathcal{M}_K$, $t$ in some compact neighborhood of $0$, and $Q(s)$ is a $d \times n$ orthonormal basis matrix for the range of $\nabla h(s)^\top$ that varies smoothly in $s$. Introduce the shorthand notation $a(s,t) = a(s + Q(s)t)$; $g(s,t)$ and $h(s,t)$ are similarly defined. Following Weyl

(1939), we have

$$\int_{\mathcal{M}_K^\varepsilon} g(x)a(x)dx = \int_K \left[\int_{\mathcal{M}_s^\varepsilon} g(s,t)a(s,t)J(s,t)dt\right] \lambda_{\mathcal{M}}(ds). \tag{S.2}$$

In (S.2), $\mathcal{M}_s^\varepsilon = \{t \in \mathcal{R}^n : \|Q(s)t\| \leq \tau(s), \|h(s,t)\| \leq \varepsilon\}$ is the $s$-section of $\mathcal{M}_K^\varepsilon$. $J(s,t) = \det(K(s,t))$, and the $(d-n) \times (d-n)$ matrix $K(s,t)$ has elements $K_{\alpha\beta}(s,t) = \mathbb{I}_{\{\alpha=\beta\}} + \sum_{i=1}^n t_i G_\alpha^\beta(i;s)$, where $\alpha, \beta \in \{1, \ldots, d-n\}$, $t_i$ is the $i$th coordinate of $t$, and $G_\alpha^\beta(i;s)$ denotes the coefficients of the second fundamental form in the $i$th direction of the normal space at $s$. While the detailed expression of $J(s,t)$ is not used in the sequel, we do need the obvious fact that $J(s,0) = \det(\iota_{d-n}) = 1$. It remains to show that

$$\varepsilon^{-n} \int_{\mathcal{M}_s^\varepsilon} g(s,t)a(s,t)J(s,t)dt \to \frac{\pi^{n/2}}{\Gamma(1+n/2)} g(s,0)a(s,0) \det\left(\nabla h(s)\nabla h(s)^\top\right)^{-1/2} \tag{S.3}$$

for each $s \in \mathcal{M}_K$, because (S.1) is a consequence of (S.3) by the Dominated Convergence Theorem.

Fix $s \in \mathcal{M}_K$. We proceed to find a region that approximates $\mathcal{M}_s^\varepsilon$ for sufficiently small $\varepsilon$'s. By the twice continuous differentiability of $h$, we have the Taylor series expansion of $h(s,t)$ at $t = 0$:

$$h(s,t) = H_0(s)t + \xi(s,t), \tag{S.4}$$

in which

$$H_t(s) = \nabla_t h(s,t) = \nabla h(s + Q(s)t)Q(s) \tag{S.5}$$

is the directional derivative of $h$ along $t$, and the remainder $\xi(s,t)$ satisfies $\|\xi(s,t)\| = o(\|t\|)$ (uniformly for $s \in \mathcal{M}_K$). Let $\hat{\mathcal{M}}_s^\varepsilon = \{t \in \mathcal{R}^n : \|H_0(s)t\| \leq \varepsilon\}$, which is an ellipsoid in $\mathcal{R}^n$. Because $H_0(s)$ is of full rank and $\mathcal{M}_K$ is compact, $\|H_0(s)t\|$ is bounded from below by a constant multiple of $\|t\|$; therefore, $\hat{\mathcal{M}}_s^\varepsilon$ is contained in the tubular neighborhood $\overline{\mathcal{T}}_K$ when $\varepsilon$ is sufficiently small. By the volume formula for ellipsoids,

$$\text{vol}\{\hat{\mathcal{M}}_s^\varepsilon\} = \int_{\hat{\mathcal{M}}_s^\varepsilon} dt = \frac{\pi^{n/2}\varepsilon^n}{\Gamma(1+n/2)} \det\left(H_0(s)^\top H_0(s)\right)^{-1/2}$$

$$= \frac{\pi^{n/2}\varepsilon^n}{\Gamma(1+n/2)} \det\left(\nabla h(s)\nabla h(s)^\top\right)^{-1/2}. \tag{S.6}$$

The last equality of (S.6) follows from

$$\det\left(H_0(s)^\top H_0(s)\right) = \det\left(Q(s)^\top Q(s)R(s)R(s)^\top Q(s)^\top Q(s)\right) = \det\left(R(s)R(s)^\top\right)$$
$$= \det\left(R(s)^\top R(s)\right) = \det\left(\nabla h(s)\nabla h(s)^\top\right), \tag{S.7}$$

in which $\nabla h(s)^\top = Q(s)R(s)$ is the expansion of $\nabla h(s)^\top$ on the basis $Q(s)$, and $H_0(s) = \nabla h(s)Q(s)$ by (S.5).

To establish (S.3), note that

$$\varepsilon^{-n}\left|\int_{\mathcal{M}_s^\varepsilon} g(s,t)a(s,t)J(s,t)dt - g(s,0)a(s,0)\int_{\hat{\mathcal{M}}_s^\varepsilon} dt\right|$$
$$\leq \underbrace{\varepsilon^{-n}\int_{\hat{\mathcal{M}}_s^\varepsilon} |g(s,t)a(s,t)J(s,t) - g(s,0), a(s,0)|\, dt}_{(I)}$$
$$+ \underbrace{\varepsilon^{-n}\int_{\hat{\mathcal{M}}_s^\varepsilon \triangle \mathcal{M}_s^\varepsilon} |g(s,t)a(s,t)J(s,t)|\, dt,}_{(II)} \tag{S.8}$$

in which $A\triangle B$ stands for the symmetric difference between sets $A$ and $B$. For reasons that $a(s,t)$, $g(s,t)$, and $J(s,t)$ are all continuous at $t=0$ and that the ellipsoid $\hat{\mathcal{M}}_s^\varepsilon$ is compact,

$$(I) \leq \sup_{t\in\hat{\mathcal{M}}_s^\varepsilon} |g(s,t)a(s,t)J(s,t)dt - g(s,0)a(s,0)| \frac{\pi^{n/2}}{\Gamma(1+n/2)}\det\left(\nabla h(s)\nabla h(s)^\top\right)^{-1/2} \to 0. \tag{S.9}$$

For (II), the Taylor expansion (S.4) allows us to find $l(\varepsilon) \leq \varepsilon \leq u(\varepsilon)$ such that $u(\varepsilon) - l(\varepsilon) = o(\varepsilon)$, $u(\varepsilon) \leq \varepsilon_0$, and $\hat{\mathcal{M}}_s^{l(\varepsilon)} \subseteq \mathcal{M}_s^\varepsilon \subseteq \hat{\mathcal{M}}_s^{u(\varepsilon)}$ for sufficiently small $\varepsilon$. We then have

$$(II) \leq \varepsilon^{-n}\int_{\hat{\mathcal{M}}_{u(\varepsilon)}^s \setminus \hat{\mathcal{M}}_{l(\varepsilon)}^s} |g(s,t)a(s,t)J(s,t)|\, dt$$
$$\leq \sup_{t\in\overline{\mathcal{T}}_K} |g(s,t)a(s,t)J(s,t)| \cdot \varepsilon^{-n}\int_{\hat{\mathcal{M}}_{u(\varepsilon)}^s \setminus \hat{\mathcal{M}}_{l(\varepsilon)}^s} dt. \tag{S.10}$$

On the right-hand side of (S.10), the supremum is bounded by continuity and compactness;

3

the remaining part satisfies

$$\varepsilon^{-n} \int_{\hat{\mathcal{M}}^s_{u(\varepsilon)} \setminus \hat{\mathcal{M}}^s_{l(\varepsilon)}} dt = \frac{(u(\varepsilon)^n - l(\varepsilon)^n)\pi^{n/2}}{\varepsilon^n \Gamma(n+1/2)} \det\left(\nabla h(s) \nabla h(s)^\top\right)^{-1/2} \to 0. \tag{S.11}$$

As such, (II) $\to 0$ as $\varepsilon \downarrow 0$. The proof of (S.1) is now complete.

## A.2 On the Ambient Space

By Prohorov's Theorem, the tightness of $\{P_\varepsilon\}$, i.e., Assumption 1 iii), guarantees that $\{P_\varepsilon\}$ contains converging subsequences. Let $\{P_{\varepsilon_i}\}$ be a subsequence such that $\varepsilon_i \downarrow 0$ and $P_{\varepsilon_i} \to P^*$ as $i \to \infty$. Take a bounded $P^*$-continuity set $B$ such that $P^*\{B\} > 0$; this is possible because a single probability measure $P^*$ is trivially tight and $P^*$-discontinuity sets are sparse. Again by tightness, we can find for each $\eta \in (0, 1/2)$ a compact set $K_\eta \subset \mathcal{X}$ such that $K_\eta \uparrow \mathcal{X}$, and that $B \cap \mathcal{M}^\varepsilon \subset \overline{\mathcal{T}}_{K_\eta}$ and $P_\varepsilon\{\overline{\mathcal{T}}_{K_\eta}\} \geq 1 - \eta$ for all sufficiently small $\varepsilon$'s. Using the shorthand notation $b(x) = a(x) \det(\nabla h(x) \nabla h(x)^\top)^{-1/2}$, we have

$$\left| \frac{\int_{B \cap \mathcal{M}} b(x) \lambda_\mathcal{M}(dx)}{\int_{K_\eta \cap \mathcal{M}} b(x) \lambda_\mathcal{M}(dx)} - P^*\{B\} \right| \leq \underbrace{\left| \frac{\int_{B \cap \mathcal{M}} b(x) \lambda_\mathcal{M}(dx)}{\int_{K_\eta \cap \mathcal{M}} b(x) \lambda_\mathcal{M}(dx)} - \frac{P_{\varepsilon_i}\{B \cap \overline{\mathcal{T}}_{K_\eta}\}}{P_{\varepsilon_i}\{\overline{\mathcal{T}}_{K_\eta}\}} \right|}_{(I)}$$

$$+ \underbrace{\left| \frac{P_{\varepsilon_i}\{B \cap \overline{\mathcal{T}}_{K_\eta}\}}{P_{\varepsilon_i}\{\overline{\mathcal{T}}_{K_\eta}\}} - P_{\varepsilon_i}\{B\} \right|}_{(II)} + \underbrace{|P_{\varepsilon_i}\{B\} - P^*\{B\}|}_{(III)}. \tag{S.12}$$

When $i$ is sufficiently large, (I) can be made arbitrarily small, say (I) $\leq \eta$, by the result having been proved in Section A.1, (II) $\leq \eta P_{\varepsilon_i}\{B\}/(1-\eta) \leq 2\eta$ for all $\eta \in (0, 1/2)$ by tightness, and (III) $\leq \eta$ due to the weak convergence of the subsequence and the Portmanteau Lemma. Altogether (S.12) $\leq 4\eta$ when $i$ is sufficiently large. The left-hand side of (S.12) then vanishes as we send $\eta$ to 0. Together with the positivity of $P^*\{B\}$ and $\int_{B \cap \mathcal{M}} b(x) \lambda_\mathcal{M}(dx)$, we have $\int_\mathcal{M} b(x) \lambda_\mathcal{M}(dx) = \lim_{\eta \downarrow 0} \int_{K_\eta \cap \mathcal{M}} b(x) \lambda_\mathcal{M}(dx) < \infty$. The integrability of $b(x)$ with respect to $\lambda_\mathcal{M}$ further allows us to apply the previous argument to any bounded $P^*$-continuity sets $B'$ with $P^*\{B'\} = 0$. We then conclude that $P_0$ and $P^*$ coincides on all bounded $P^*$-continuity sets.

Let $\mathcal{P} = \{$all bounded $P^*$-continuity sets$\}$: $\mathcal{P}$ is closed under finite intersections and thus a $\pi$-system. Dynkin's $\pi$-$\lambda$ Theorem therefore extends the equity of $P_0$ and $P^*$ on all

Borel sets. It further implies that the full sequence $P_\varepsilon$ converges weakly to $P_0$ as the converging subsequence $\{P_{\varepsilon_i}\}$ is arbitrarily chosen. We now conclude the proof of Theorem 1.

# Appendix B

# Proof of Lemma 1

Under Assumption 2, we have $\nabla_\theta G(u, \theta)^\top [G(u, \theta) - y] = 0$ for all $(u^\top, \theta^\top)^\top \in \mathcal{G}(y)$, which is the first-order condition for minimizing $\|G(u, \theta) - y\|^2/2$ with respect to $\theta$. By Assumption 2 and the Implicit Function Theorem (Rudin; 1964, Chapter 9), there exists a unique, twice continuously differentiable function $\hat{\theta}(y, u)$ such that

$$\nabla_\theta G(u, \hat{\theta}(y, u))^\top [G(u, \hat{\theta}(y, u)) - y] = 0 \tag{S.13}$$

for all $u$ in some neighborhood of $\mathcal{U}(y)$. The isomorphism between $\mathcal{G}(y)$ and $\mathcal{U}(y)$ follows from the uniqueness of $\hat{\theta}(y, u)$ for all $u \in \mathcal{U}(y)$.

Project $G(u, \theta) - y$ to the range and null spaces of $\nabla_\theta G(u, \theta)$, respectively:

$$\begin{aligned} G(u, \theta) - y &= \nabla_\theta G(u, \theta) \left(\nabla_\theta G(u, \theta)^\top \nabla_\theta G(u, \theta)\right)^{-1} \nabla_\theta G(u, \theta)^\top [G(u, \theta) - y] \\ &\quad + \overline{\nabla_\theta G}(u, \theta) \overline{\nabla_\theta G}(u, \theta)^\top [G(u, \theta) - y], \end{aligned} \tag{S.14}$$

in which the $n \times (n - q)$ matrix $\overline{\nabla_\theta G}(u, \theta)$ can be any orthogonal complement of $\nabla_\theta G(u, \theta)$ with orthonormal columns. Upon replacing $\theta$ by $\hat{\theta}(y, u)$, the first term on the right-hand side of (S.14) vanishes due to (S.13). Hence, for any $\epsilon \geq 0$, $\|G(u, \hat{\theta}(y, u)) - y\| = \epsilon$ is equivalent to $\|h(u)\| = \epsilon$ with

$$h(u) = \overline{\nabla_\theta G}(u, \hat{\theta}(y, u))^\top [G(u, \hat{\theta}(y, u)) - y], \tag{S.15}$$

which justifies (23). $h(u)$ is twice continuously differentiable, which follows from the three-time continuous differentiability of $G$ and the full-rank assumption for $\nabla_\theta G$. Moreover, note that

$$\nabla h(u) = \overline{\nabla_\theta G}(u, \hat{\theta}(y, u))^\top \nabla_u G(u, \hat{\theta}(y, u)) \tag{S.16}$$

for all $u \in \mathcal{U}(y)$. As $\nabla_u G(u, \hat{\theta}(y, u))$ is assumed to have a full row rank, the rank of $\nabla h(u)$ equals to $n - q$, and thus $h(u)$ is a submersion.

Before establishing (25) that translates between the intrinsic measures $\lambda_{\mathcal{G}(y)}$ and $\lambda_{\mathcal{U}(y)}$, we briefly introduce the notion of local parameterization. For a differentiable submanifold $\mathcal{M} \subset \mathcal{R}^d$ of dimension $l$, $l < d$, we are guaranteed to have a collection of local

6

diffeomorphisms—called *coordinate charts*—that map elements of an open cover of $\mathcal{M}$ to subsets of an $l$-dimensional Euclidean space. The inverse of the diffeomorphic map is often referred to as a *local parameterization* of $\mathcal{M}$. Local integration over a set on $\mathcal{M}$ can be computed by fixing a coordinate chart that covers this set and performing the usual Lebesgue integration on the Euclidean space pushed forward by the diffeomorphism. Local integration is extended to global integration via a smooth partition of unity: a collection of weighting functions that allow one piece together the local integrals. Precise definitions of the aforementioned terms can be found in, e.g., Lee (2013).

Let $u(\omega)$ denote a smooth local parameterization of the manifold $\mathcal{U}(y)$. We have

$$\lambda_{\mathcal{U}(y)}(du) = \det\left(\nabla u(\omega)^\top \nabla u(\omega)\right)^{1/2} d\omega$$

$$= \frac{\det\left(\nabla u(\omega)^\top \nabla u(\omega)\right)^{1/2} \det\left(\nabla u(\omega)^\top \nabla u(\omega) + \nabla u(\omega)^\top \nabla_u \hat\theta(y, u(\omega))^\top \nabla_u \hat\theta(y, u(\omega)) \nabla u(\omega)\right)^{1/2}}{\det\left(\nabla u(\omega)^\top \nabla u(\omega) + \nabla u(\omega)^\top \nabla_u \hat\theta(y, u(\omega))^\top \nabla_u \hat\theta(y, u(\omega)) \nabla u(\omega)\right)^{1/2}} d\omega$$

$$= \frac{\det\left(\nabla u(\omega)^\top \nabla u(\omega)\right)^{1/2}}{\det\left(\nabla u(\omega)^\top \nabla u(\omega) + \nabla u(\omega)^\top \nabla_u \hat\theta(y, u(\omega))^\top \nabla_u \hat\theta(y, u(\omega)) \nabla u(\omega)\right)^{1/2}} \lambda_{\mathcal{G}(y)}(du, d\theta). \quad \text{(S.17)}$$

In (S.17), we have

$$\nabla_u \hat\theta(y, u) = -\left[\nabla_\theta G(u, \hat\theta(y, u))^\top \nabla_\theta G(u, \hat\theta(y, u))\right]^{-1} \nabla_\theta G(u, \hat\theta(y, u))^\top \nabla_u G(u, \hat\theta(y, u)), \tag{S.18}$$

for $u \in \mathcal{U}(y)$, which is obtained by differentiating (S.13) with respect to $u$.

Our remaining task is to show that the ratio in the last line of (S.17) equals to $D(u, \theta)^{-1/2}$. Dependencies on $u$ and $\theta$ are dropped for the rest of the appendix to simplify the notation. By the Matrix Determinant Lemma and (S.18),

$$\det\left(\nabla u^\top \nabla u + \nabla u^\top \nabla_u \hat\theta^\top \nabla_u \hat\theta \nabla u\right)$$
$$= \det\left(\nabla u^\top \nabla u\right) \cdot \det\left(\iota_q + \nabla_u \hat\theta \nabla u \left(\nabla u^\top \nabla u\right)^{-1} \nabla u^\top \nabla_u \hat\theta^\top\right) = \det\left(\nabla u^\top \nabla u\right)$$
$$\cdot \underbrace{\det\left(\iota_q + \left(\nabla_\theta G^\top \nabla_\theta G\right)^{-1} \nabla_\theta G^\top \nabla_u G \nabla u \left(\nabla u^\top \nabla u\right)^{-1} \nabla u^\top \nabla_u G^\top \nabla_\theta G \left(\nabla_\theta G^\top \nabla_\theta G\right)^{-1}\right)}_{=: \tilde{D}}.$$

(S.19)

7

It further suffices to show that $\tilde{D} = D$. Note that $\tilde{D}$ depends on $\nabla u$ only through the projection matrix $\nabla u (\nabla u^\top \nabla u)^{-1} \nabla u^\top$, so $\tilde{D}$ is invariant to difference choices of $\nabla u$. We proceed to set

$$\nabla u = (K : L), \tag{S.20}$$

in which $K$ is an $m \times (m-n)$ orthogonal complement of $\nabla_u G^\top$ with orthonormal columns (i.e., $K^\top \nabla_u G^\top = 0$ and $K^\top K = \iota_{m-n}$) and $L = \nabla_u G^\top (\nabla_u G \nabla_u G^\top)^{-1} \nabla_\theta G$ with dimension $m \times q$.[3] To see that (S.20) is a valid choice of $\nabla u$, we first note that $(K : L)$ has a full column rank equal to $m - n + q$: This is because $L$ is of full column rank by Assumption 2 ii) and $K^\top L = 0$. Moreover,

$$\overline{\nabla_\theta G}^\top \nabla_u G (K : L) = \overline{\nabla_\theta G}^\top (0 : \nabla_\theta G) = 0, \tag{S.21}$$

so $(K : L)$ is perpendicular to the $m \times (n-q)$ matrix $\nabla_u G^\top \overline{\nabla_\theta G}$ that spans the normal space of $\mathcal{U}(y)$. Our choice of $\nabla u$ allows us to simplify $\tilde{D}$ to

$$\tilde{D} = \det\left(\iota_q + (L^\top L)^{-1}\right) = \det\left(\iota_q + [\nabla_\theta G^\top (\nabla_u G \nabla_u G^\top)^{-1} \nabla_\theta G]^{-1}\right) = D. \tag{S.22}$$

The proof of the lemma is complete.

---

[3] Assumption 2 ii) implies that $m \geq n$; if $m = n$, simply remove $K$ from (S.20).

# Appendix C

# Proof of Propositions 1 and 2

## C.1 Proof of Proposition 1

(33) is a direct consequence of Theorem 1. To establish (34), we apply the Matrix Determinant Lemma twice to (26):

$$D = \frac{\det\left(\iota_q + \nabla_\theta G^\top \left(\nabla_u G \nabla_u G^\top\right)^{-1} \nabla_\theta G\right)}{\det\left(\nabla_\theta G^\top \left(\nabla_u G \nabla_u G^\top\right)^{-1} \nabla_\theta G\right)}$$

$$= \frac{\det\left(\nabla_u G \nabla_u G^\top + \nabla_\theta G \nabla_\theta G^\top\right)}{\det\left(\nabla_u G \nabla_u G^\top\right) \det\left(\nabla_\theta G^\top \left(\nabla_u G \nabla_u G^\top\right)^{-1} \nabla_\theta G\right)}. \quad (S.23)$$

(34) immediately follows from (S.23) and the change of measure in Lemma 1.

To arrive at the last statement of Proposition 1, we apply the smooth co-area formula (Chavel; 2006, Section III.8): For any measurable $B \subseteq \Theta$,

$$\int_{\hat{\theta}(y,u)\in B} \tilde{f}_B(u) \lambda_{\mathcal{U}(y)}(du) = \int_B \left[\int \frac{\tilde{f}_B(u)}{\det\left(\nabla_u \hat{\theta}|_{T_u\mathcal{U}(y)} \nabla_u \hat{\theta}^\top|_{T_u\mathcal{U}(y)}\right)^{1/2}} \lambda_{\mathcal{G}_\theta(y)}(du)\right] d\theta, \quad (S.24)$$

in which $\nabla_u \hat{\theta}^\top|_{T_u\mathcal{U}(y)}$ stands for the projection of the $m \times q$ dimensional $\nabla_u \hat{\theta}^\top$ onto the $(m - n + q)$-dimensional tangent space of $\mathcal{U}(y)$ at $u$. Choosing (S.20) as the basis for $T_u\mathcal{U}(y)$, we have

$$\det\left(\nabla_u \hat{\theta}|_{T_u\mathcal{U}(y)} \nabla_u \hat{\theta}^\top|_{T_u\mathcal{U}(y)}\right) = \det\left(\nabla_\theta G^\top \left(\nabla_u G \nabla_u G^\top\right)^{-1} \nabla_\theta G\right)^{-1}. \quad (S.25)$$

Therefore, the bracketed term on the right-hand side of (S.24) can be further written as

$$\int \frac{\tilde{f}_B(u)}{\det\left(\nabla_u \hat{\theta}|_{T_u\mathcal{U}(y)} \nabla_u \hat{\theta}^\top|_{T_u\mathcal{U}(y)}\right)^{1/2}} \lambda_{\mathcal{G}_\theta(y)}(du) \propto \pi(\theta) \left[\int \frac{\rho(u)}{\det\left(\nabla_u G \nabla_u G^\top\right)^{1/2}} \lambda_{\mathcal{G}_\theta(y)}(du)\right], \quad (S.26)$$

which is proportional to the posterior density due to (31).

## C.2 Proof of Proposition 2

(35) is a direct consequence of Theorem 1. To establish (36), we derive an

expression for $\det\left(\overline{\nabla_\theta G}^\top \nabla_u G \nabla_u G^\top \overline{\nabla_\theta G}\right)$ that does not explicitly involve the orthogonal complement $\overline{\nabla_\theta G}$. Let $A = (\nabla_\theta G : \overline{\nabla_\theta G})$, which is a full-rank square matrix with dimension $n \times n$. Applying the Schur Determinant Identity to $A^\top \nabla_u G \nabla_u G^\top A$, we obtain

$$\det\left(A^\top \nabla_u G \nabla_u G^\top A\right) = \det(A)^2 \det\left(\nabla_u G \nabla_u G^\top\right) = \det(\overline{\nabla_\theta G}^\top \nabla_u G \nabla_u G^\top \overline{\nabla_\theta G})$$
$$\cdot \det\left(\nabla_\theta G^\top \nabla_u G \left[\iota_m - \nabla_u G^\top \overline{\nabla_\theta G}(\overline{\nabla_\theta G}^\top \nabla_u G \nabla_u G^\top \overline{\nabla_\theta G})^{-1} \overline{\nabla_\theta G}^\top \nabla_u G\right] \nabla_u G^\top \nabla_\theta G\right)$$
$$= \det\left(\overline{\nabla_\theta G}^\top \nabla_u G \nabla_u G^\top \overline{\nabla_\theta G}\right) \det\left(\nabla_\theta G^\top \nabla_u G \nabla u \left(\nabla u^\top \nabla u\right)^{-1} \nabla u^\top \nabla_u G^\top \nabla_\theta G\right)$$
$$= \frac{\det\left(\overline{\nabla_\theta G}^\top \nabla_u G \nabla_u G^\top \overline{\nabla_\theta G}\right) \det\left(\nabla_\theta G^\top \nabla_\theta G\right)^2}{\det\left(\nabla_\theta G^\top \left(\nabla_u G \nabla_u G^\top\right)^{-1} \nabla_\theta G\right)}, \qquad (S.27)$$

which follows from (S.20) and (S.21). (S.27) and the equality $\det(A)^2 = \det(\nabla_\theta G^\top \nabla_\theta G)$ imply that

$$\det\left(\overline{\nabla_\theta G}^\top \nabla_u G \nabla_u G^\top \overline{\nabla_\theta G}\right) = \frac{\det\left(\nabla_u G \nabla_u G^\top\right) \det\left(\nabla_\theta G^\top \left(\nabla_u G \nabla_u G^\top\right)^{-1} \nabla_\theta G\right)}{\det\left(\nabla_\theta G^\top \nabla_\theta G\right)}. \qquad (S.28)$$

(36) is then deduced from (S.23), (S.28), and Lemma 1.

Similar to (S.24)–(S.26), the last statement of Proposition 2 also follows directly from the smooth co-area formula:

$$\int_{\hat{\theta}(y,u) \in B} \tilde{f}_F(u) \lambda_{\mathcal{U}(y)}(du) = \int_B \left[\int \frac{\tilde{f}_F(u)}{\det\left(\nabla_u \hat{\theta}|_{T_u \mathcal{U}(y)} \nabla_u \hat{\theta}^\top|_{T_u \mathcal{U}(y)}\right)^{1/2}} \lambda_{\mathcal{G}_\theta(y)}(du)\right] d\theta. \qquad (S.29)$$

The bracketed term in (S.29) gives the density of $\hat{\theta}(y, u)$, which is proportional to (37) as a result of (S.28).

## Appendix D

## Proof of Proposition 3

For succinctness, we treat $E$ as an $IJ \times 1$ vector throughout this proof, replacing the notation $\text{vec}(E)$ in the main text. The repeated-measures ANOVA model can then be expressed in matrix form as

$$Y = (1_J \otimes \iota_I)\mu + \sigma_z(\iota_J \otimes 1_I)Z + \sigma_e E$$
$$= W(Z)\beta + \sigma_e E. \tag{S.30}$$

in which[4] $W(z) = (1_J \otimes \iota_I : (\iota_J \otimes 1_I)z) \in \mathcal{R}^{I+1}$ and $\beta = (\mu^\top, \sigma_z)^\top$. Also let $\overline{W}(z) \in \mathcal{R}^{IJ-I-1}$ be an orthogonal complement of $W(z)$ with orthonormal columns, and $r(y,z) = \overline{W}(z)\overline{W}(z)^\top y$ be the projection of $y$ onto the null space of $W(z)$—equivalently, the residual after regressing $y$ on $W(z)$.

We proceed to characterize the set

$$C_\varepsilon(y,z) = \{e \in \mathcal{R}^{IJ} : \min_{\beta,\sigma_e} \|W(z)\beta + \sigma_e e - y\| \leq \varepsilon\}, \tag{S.31}$$

There are two cases to consider. First, if $\|r(y,z)\| \leq \varepsilon$, then $C_\varepsilon(y,z) = \mathcal{R}^{IJ}$ because the minimum of $\|W(z)\beta - y\|$ (i.e., fixing $\sigma_e$ at 0) is already no greater than $\varepsilon$. Second, if $\|r(y,z)\| > \varepsilon$, then the least-square solution of $\sigma_e$ (i.e., $\hat{\sigma}_e$) corresponding to $e \in C_\varepsilon(y,z)$ must be non-zero. We claim that $C_\varepsilon(y,z)$ in the second case is equivalent to

$$\tilde{C}_\varepsilon(y,z) = \{[r(y,z)+e_1]\alpha + e_2 : e_1 \in \mathfrak{n}(W(z)), \|e_1\| \leq \varepsilon, \alpha \neq 0, e_2 \in \mathfrak{r}(W(z))\}, \tag{S.32}$$

in which $\mathfrak{r}(A)$ and $\mathfrak{n}(A)$ denote the range and null space for the columns of $A$. To see this, first take $e \in C_\varepsilon(y,z)$, which satisfies $y = W(z)\hat{\beta} + \hat{\sigma}_e e + \varrho$ with some $\varrho \in \mathcal{R}^{IJ}$ such that $\|\varrho\| \leq \varepsilon$. Because $\hat{\sigma}_e \neq 0$, we have

$$e = \overline{W}(z)\overline{W}(z)^\top \left(\frac{y-\varrho}{\hat{\sigma}_e}\right) + [\iota_{IJ} - \overline{W}(z)\overline{W}(z)^\top]\left(\frac{y-\varrho}{\hat{\sigma}_e}\right) - \frac{W(z)\hat{\beta}}{\hat{\sigma}_e}$$
$$= \underbrace{[r(y,z) - \overline{W}(z)\overline{W}(z)^\top \varrho]}_{e_1} \underbrace{\sigma_e^{-1}}_{\alpha} + \underbrace{[\iota_{IJ} - \overline{W}(z)\overline{W}(z)^\top]\left(\frac{y-\varrho}{\hat{\sigma}_e}\right) - \frac{W(z)\hat{\beta}}{\hat{\sigma}_e}}_{e_2}, \tag{S.33}$$

---
[4]$W(z)$ is rank deficient (with rank $I < I+1$) when $z$ is a multiple of $1_J$.

11

which can be identified as an element in $\tilde{C}_\varepsilon(y, z)$. Conversely, take $e \in \tilde{C}_\varepsilon(y, z)$ so that $e = (r(y, z) + e_1)\alpha + W(z)\beta_1$ for some $\alpha \neq 0$, $\beta_1 \in \mathcal{R}^{I+1}$, and $e_1 \in \mathfrak{n}(W(z))$ such that $\|e_1\| \leq \varepsilon$. As $r(y, z) = y - W(z)\beta_2$ for some $\beta_2 \in \mathcal{R}^{I+1}$, we let $\beta = \beta_1 - \beta_2\alpha$ and therefore have

$$\left\|y + \frac{W(z)\beta}{\alpha} - \frac{e}{\alpha}\right\| = \|-e_1\| \leq \varepsilon, \tag{S.34}$$

which implies $e \in C_\varepsilon(y, z)$. Geometrically, $\tilde{C}_\varepsilon(y, z)$ in (S.32) is the Cartesian product of $\mathfrak{r}(W(z))$ and a double (spherical) cone in $\mathfrak{n}(W(z))$ centered at the origin. This is because in (S.32) $r(y, z) + e_1$ with $\|e_1\| \leq \varepsilon$ falls within an $\ell_2$-ball around $r(y, z)$ with radius $\varepsilon < r(y, z)$; therefore, points that are multiples of $r(y, z) + e_1$ form a spherical cone with central angle $\sin^{-1}(\varepsilon/\|r(y, z)\|)$.

Our final task is to find compact sets $K \in \mathcal{R}^J$ and $L \in \mathcal{R}^{IJ}$ such that the ratio

$$\frac{P\{\{E \in C_\varepsilon(y, Z) \cap L\} \cap \{Z \in K\}\}}{P\{E \in C_\varepsilon(y, Z)\}} \tag{S.35}$$

can be made arbitrarily close to 1. Taking advantage of the spherical symmetry in our setup, let $K \in \mathcal{R}^J$ and $L \in \mathcal{R}^{IJ}$ be closed $\ell_2$-balls centered at the origin. Because $E$ follows a spherical distribution independent of $Z$ and $C_\varepsilon(y, z)$ is spherically symmetric,

$$P\{E \in C_\varepsilon(y, z) \cap L \mid Z = z\} = P\{E \in L\} \cdot \varpi(y, z, \varepsilon), \tag{S.36}$$

where $\varpi(y, z, \varepsilon) \in (0, 1]$ is equal to 1 if $\|r(y, z)\| \leq \varepsilon$ and otherwise is a monotonically increasing function of the spherical cone's central angle. Because $Z$ also follows a spherical distribution, let $Z = R_z V_z$ where $V_z \in \mathcal{R}^J$ is uniform on the unit sphere and $R_z > 0$ is independent of $V_z$. The earlier geometric analysis reveals that the central angle of the spherical cone depends only on $V_z$ but not $R_z$. It follows that

$$P\{\{E \in C_\varepsilon(y, Z) \cap L\} \cap \{Z \in K\}\} = \int_K P\{E \in C_\varepsilon(y, z) \cap L \mid Z = z\} P(dz)$$
$$= P\{E \in L\} \int_K \varpi(y, z, \varepsilon) P(dz) = P\{E \in L\} \int \varpi(y, v_z, \varepsilon) dP(dv_z) \int_0^{\gamma(K)} dP(dr_z), \tag{S.37}$$

in which $v_z$ and $r_z$ are respective realizations of $V_z$ and $R_z$, and $\gamma(K)$ denotes the radius of

12

$K$. Similarly,

$$P\{E \in C_\varepsilon(y, Z)\} = \int P\{E \in C_\varepsilon(y, z) \mid Z = z\} P(dz)$$
$$= \int \varpi(y, z, \varepsilon) P(dz) = \int \varpi(y, v_z, \varepsilon) dP(dv_z). \tag{S.38}$$

Hence, the ratio of (S.37) over (S.38) is

$$\frac{P\{\{E \in C_\varepsilon(y, Z) \cap L\} \cap \{Z \in K\}\}}{P\{E \in C_\varepsilon(y, Z)\}} = P\{E \in L\}P\{Z \in K\}, \tag{S.39}$$

which is constant in $\varepsilon$ and can be made arbitrarily close to 1.

# Appendix E

# Computational Complexity for Repeated-Measures ANOVA

## E.1 Evaluating the Fiducial Density

When $I, J > 1$, $IJ \geq I + 2$. The matrix $\nabla_u G \nabla_u G^\top + \nabla_\theta G \nabla_\theta G^\top$ is then a low-rank modification to the matrix $\nabla_u G \nabla_u G^\top$, which allows us to use well-known linear algebraic results such as the Woodbury identity and Matrix Determinant Lemma to lessen the computational burden.[5]

Because the schoolbook complexity for computing $\det(\nabla_\theta G^\top \nabla_\theta G)$, which appears as the first determinant term on the right-hand side of (36), is already $O(I^3 J)$, we focus on the second determinant term, i.e., (38) after applying the Matrix Determinant Lemma. Note that

$$\nabla_u G \nabla_u G^\top = \iota_J \otimes \underbrace{(\sigma_e^2 \iota_I + \sigma_z^2 1_I 1_I^\top)}_{\Omega}. \tag{S.40}$$

(S.40) has a repetitive block-diagonal structure. The $I \times I$ diagonal block $\Omega = \sigma_e^2 \iota_I + \sigma_z^2 1_I 1_I^\top$ is a rank-one modification to a diagonal matrix: $\Omega^{-1} = \sigma_e^{-2} \iota_I - \sigma_e^{-4} \sigma_z^2 (1 + I \sigma_e^{-2} \sigma_z^2)^{-1} 1_I 1_I^\top$ by the Woodbury formula and $\det(\Omega) = \sigma_e^{2I}(1 + I \sigma_e^{-2} \sigma_z^2)$ by the Matrix Determinant Lemma. Therefore, solving the linear system $(\iota_J \otimes \Omega) x = b$ for $x, b \in \mathcal{R}^{IJ}$ takes only $O(IJ)$ flops rather than $O(I^3 J^3)$ flops that would have been needed for an unstructured left-hand side matrix. This further reduces the computation of $\nabla_\theta G (\nabla_u G \nabla_u G^\top)^{-1} \nabla_\theta G$ to $O(I^3 J)$ flops assuming the schoolbook complexity for matrix multiplication and determinant calculation. Because the first determinant in (38) takes also $O(I^3 J)$ flops to evaluate, the overall complexity of evaluating the fiducial density is $O(I^3 J)$ rather than $O(I^3 J^3)$.

## E.2 Manifold MCMC Update

The complexity of the manifold RWM update is determined by two operations: finding an orthonormal basis for the null space of $(\nabla_u G : \nabla_\theta G)^\top$ and projecting a point back to the manifold (i.e., Algorithm 2). The projection step solves linear equations with left-hand side matrices of the form $\nabla_u G(u, \theta) \nabla_u G(u', \theta')^\top + \nabla_\theta G(u, \theta) \nabla_\theta G(u', \theta')^\top$, where $u, u' \in \mathcal{R}^{(I+1)J}$, $\theta, \theta' \in \mathcal{R}^{I+2}$: An argument similar to the previous paragraph shows that its complexity is $O(I^3 J)$.

---

[5] For notational succinctness, we again suppress the dependency on $u$ and $\theta$.

An orthonormal basis matrix of the null space is routinely obtained via a full QR decomposition: For $(\nabla_u G : \nabla_\theta G)^\top$, it takes $O(I^3 J^3)$ flops. Nevertheless, we can take advantage of the fact that $\nabla_u G$ contains a $IJ \times IJ$ diagonal block. In particular, it can be straightforwardly verified that

$$\begin{pmatrix} -\sigma_e \iota_{I+J+2} \\ \iota_J \otimes \sigma_z 1_I : \nabla_\theta G \end{pmatrix} \quad (S.41)$$

is an orthogonal complement of $(\iota_J \otimes \sigma_z 1_I : \nabla_\theta G : \sigma_e \iota_{IJ})^\top$, which becomes $(\nabla_u G : \nabla_\theta G)^\top$ after a suitable permutation of rows. The remaining task is to orthogonalizing and normalizing the columns of (S.41), which amounts to QR-factorizing $(\iota_J \otimes \sigma_z 1_I : \nabla_\theta G)$ because the diagonal block $-\sigma_e \iota_{I+J+2}$ already has orthogonal columns. Note that the projection matrix corresponding to $\iota_J \otimes \sigma_z 1_I$ is $\iota_J \otimes I^{-1} 1_I 1_I^\top$, which again has a repetitive block-diagonal structure. It then suffices to first project $\nabla_\theta G$ to the null space of $\iota_J \otimes \sigma_z 1_I$ and then apply a QR decomposition, each of which takes only $O(I^3 J)$ flops.



\end{document}